\documentclass[12pt,A4,leqno]{amsart}
\usepackage{amsfonts}
\usepackage{mathrsfs}
\usepackage[T1]{fontenc}
\usepackage{amssymb,amscd}
\usepackage{upgreek}
\usepackage{color}
\usepackage{epsf}
\usepackage{graphicx}
\usepackage{extarrows}
\usepackage[curve]{xypic}

\theoremstyle{plain}
\newtheorem{thm}{Theorem}[section]
\newtheorem{lem}[thm]{Lemma}
\newtheorem{prop}[thm]{Proposition}
\newtheorem{cor}[thm]{Corollary}

\theoremstyle{definition}
\newtheorem{defi}[thm]{Definition}

\newtheorem{rem}[thm]{Remark}

\newcommand{\Q}{\mathbb Q}
\newcommand{\R}{\mathbb R}
\newcommand{\Z}{\mathbb Z}

\newcommand{\nn}{\vskip 0.2cm}
\newcommand{\n}{\vskip 0.1cm}

\makeatletter
\renewcommand{\@secnumfont}{\bfseries}
\makeatother

\makeatletter
\def\section{%
  \@startsection{section}{1}
    {\z@}
    {2.0ex plus 0.8ex minus .1ex}
    {1.0ex plus .2ex}
    {\bfseries\large\centering\MakeUppercase}%
}
\makeatother

\begin{document}

\title [\ ] {On equivariantly formal $2$-torus manifolds}

\author{Li Yu}
\address{Department of Mathematics, Nanjing University, Nanjing, 210093, P.R.China.
  }
 \email{yuli@nju.edu.cn}




\thanks{
 This work is partially supported by 
 National Natural Science Foundation of China (grant no.11871266) and 
 the PAPD (priority academic program development) of Jiangsu higher education institutions.}

\begin{abstract}
 A $2$-torus manifold is a closed connected smooth $n$-manifold with a
non-free effective smooth $\Z^n_2$-action. In this paper,
 we prove that a $2$-torus manifold is equivariantly formal if and only if the $\Z^n_2$-action is locally standard and every face of its orbit space (including the whole orbit space) is mod $2$ acyclic. Our study is parallel to the study 
 of torus manifolds with vanishing odd-degree cohomology
 by M.~Masuda and T.~Panov in~\cite{MasPanov06}.  As an application, we determine when such kind of $2$-torus manifolds
 can have regular $\mathrm{m}$-involutions (i.e. involutions with only isolated fixed points of the maximum possible number). 
  \end{abstract}

\maketitle

 \section{Introduction}
  
  Let $G$ be a compact Lie group and $BG$ be the classifying space of $G$. For a $G$-space $X$,
   the \emph{$G$-equivariant cohomology} of $X$ with coefficients in a field $\mathbf{k}$ is the singular cohomology of the Borel construction $X_G$ (see~\cite{Borel60}) 
   $$H^*_G(X;\mathbf{k}):=H^*(X_G;\mathbf{k}).$$
   
   There is a natural fibration $X\rightarrow X_G \rightarrow BG$ associated to $X_G$ called the \emph{Borel fibration}. If
  the inclusion of the fiber $\iota_X: X\rightarrow X_G$ induces a surjection on cohomology
$\iota_X^* : H^*_G(X;\mathbf{k})\rightarrow H^*(X;\mathbf{k})$, 
 $X$ is called called (cohomologically) \emph{equivariantly formal} over $\mathbf{k}$.
 This term was coined in 1998 by Goresky, Kottwitz, and
MacPherson in~\cite{GKM98}. But this condition had already been studied by A.~Borel in ~\cite[\S\,4]{Borel53} and~\cite[Ch.\,XII]{Borel60} where $X$ is called \emph{totally non-homologous to zero in} $X_G$
(also see~\cite[Ch.\,VII]{Bred72}). \n

For some special groups $G$ shown below, the equivariant formality of a $G$-action can be interpreted in some other ways (see~\cite[\S\,4]{Borel53}, \cite[Ch.\,3]{AllPuppe93} and~\cite[Sec.\,4]{AllFranzPuppe21}).\n

 \begin{itemize}
\item 
 When $BG$ is simply connected (e.g. $G$ is a torus $T^r=(S^1)^r$),  $X$ is equivariantly formal if and only if the Serre spectral sequence of the Borel fibration of $X$ degenerates at the $E_2$ stage.
 \n
 \item When $G$ is the $p$-torus $\Z^r_p$ ($p$ is prime), $X$ being equivariantly formal is equivalent to either one of the following conditions.
 \begin{itemize}
  \item[(i)] The Serre spectral sequence with $\Z_p$-coefficients of the Borel fibration of $X$  degenerates at the $E_2$ stage and the induced action of $\Z^r_p$ on $H^*(X;\Z_p)$ is trivial.\n
  \item[(ii)] $H^*_{\Z^r_p}(X;\Z_p)\cong H^*(X;\Z_p)\otimes H^*(B\Z^r_p;\Z_p)$ is a free $H^*(B\Z^r_p;\Z_p)$-module.
  \end{itemize}
 \end{itemize}
 
Due to the above fact, we call a $\Z^r_p$-action on $X$ \emph{weakly equivariantly formal} if we only assume that the Serre spectral sequence (with $\Z_p$-coefficients) of the Borel fibration of $X$ degenerates at the $E_2$ stage. So an equivariantly formal
$\Z^r_p$-action is always weakly equivariantly formal.\n

 When $G=T^r$ or $\Z^r_2$ and $\mathbf{k}=\Q$ or $\Z_2$ respectively, there is another equivalent
  description of equivariantly formal $G$-actions
    given by the so called
  ``Atiyah-Bredon sequence'' (see Bredon~\cite{Bred74} and Franz-Puppe~\cite{FranzPuppe07} for the $T^r$ case, and Allday-Franz-Puppe~\cite{AllFranzPuppe21} for the $\Z^r_2$ case).
 In addition, there are many sufficient conditions for
  a $T^r$-action to be equivariantly formal (for example: 
  vanishing of odd-degree cohomology, all homology classes being representable by $T^r$-invariant cycles, etc.).  \n

 Equivariantly formal $G$-spaces provide many nice examples in geometry and topology.
Some of them are:
  \begin{itemize}
  \item Smooth compact toric varieties.\n  
  \item Hamiltonian $G$-actions on symplectic manifolds which have moment maps (see Atiyah-Bott~\cite{AtiyBott84} and Jeffrey-Kirwan~\cite{JefKirw95}).\n  
  \item Quasitoric manifolds and small covers defined in Davis-Januszkiewicz~\cite{DaJan91}.\n  
  \item Torus manifolds with vanishing odd degree cohomology (see Masuda-Panov~\cite{MasPanov06}).
  \end{itemize}

   In addition, when $G=T^r$ or
 $(\Z_p)^r$, the following theorem gives us an easy way to recognize equivariantly formal $G$-actions.
 
   \begin{thm}[see Theorem (3.10.4) in Allday-Puppe~\cite{AllPuppe93}]
    \label{thm:AllPuppe}  
       Let $G=T^r$ or $(\Z_p)^r$ where $p$ is a prime
and $\mathbf{k}=\Q$ or $\Z_p$ respectively.
Let $X$ be a paracompact $G$-space with only finitely many orbit types and $\dim_{\mathbf{k}} H^*(X;\mathbf{k})<\infty$. Then the fixed point set
$X^G$ always satisfies 
$$\dim_{\mathbf{k}} H^*(X^G;\mathbf{k}) \leq \dim_{\mathbf{k}} H^*(X;\mathbf{k})$$
 where the equality holds if and only if  
 $X$ is equivariantly formal over $\mathbf{k}$. Here $\dim_{\mathbf{k}} H^*(X;\mathbf{k})$ denotes the sum of the rank of the cohomology groups of $X$ in all dimensions over $\mathbf{k}$.
    \end{thm}

  A very special case is when $G=\Z_2$ and
  $X^{\Z_2}$ consists only of isolated points. 
  By Theorem~\ref{thm:AllPuppe}, we have
   \begin{equation} \label{Equ:FixPointSet}
      |X^{\Z_2}| = \dim_{\Z_2}H^*(X^{\Z_2};\Z_2) \leq   \dim_{\Z_2} H^*(X;\Z_2)
      \end{equation}
      
 Such a $\Z_2$-action on $X$ is equivariantly formal if and only if the number of the fixed points reaches the maximum, i.e.
$|X^{\Z_2}| = \dim_{\Z_2}H^*(X;\Z_2)$. In this case, 
the involution determined by the $\Z_2$-action is called an \emph{$\mathrm{m}$-involution} on $X$ (this term was named by Puppe~\cite{Puppe01}). \n

 There is an interesting relation between $\mathrm{m}$-involutions on closed manifolds and binary codes. It was shown in~\cite{Puppe01} that one can obtain 
a self-dual binary code from any $\mathrm{m}$-involution on an odd-dimensional closed manifold. This motivates the study in Chen-L\"{u}-Yu~\cite{ChenLuYu18} on the $\mathrm{m}$-involutions on a special kind of closed  manifolds called \emph{small covers} (see~\cite{DaJan91}).     
  In this paper, we want to study a more general type of closed manifolds with $2$-torus actions defined below.\n
  
  \begin{defi}[see L\"{u}-Masuda~\cite{MaLu08}]
   A \emph{$2$-torus manifold} is a closed connected smooth $n$-manifold $W$ with a
non-free effective smooth action of $\Z^n_2$. 
For such a manifold $W$, since $\dim(W)=n=\mathrm{rank}(\Z^n_2)$ and the $\Z^n_2$-action is effective, the fixed point set $W^{\Z^n_2}$ must be discrete. Then since $W$ is compact, $W^{\Z^n_2}$ is a finite set of isolated points (if not empty).  Note that we require all $2$-torus manifolds to be connected in this paper. 
\begin{itemize}
 
 \item For brevity, we call a $2$-torus manifold $W$ \emph{equivariantly formal} or \emph{weakly equivariantly formal} if the $\Z^n_2$-action on $W$ is so, respectively.
\n

\item
We call $W$ \emph{locally standard}
if for every point $x \in W$, there is a $\Z^n_2$-invariant neighborhood $V_x$ of $x$ such that $V_x$
is equivariantly homeomorphic to an invariant open subset of a real $n$-dimensional faithful linear representation space of $\Z^n_2$. An equivalently way to describe such a neighborhood $V_x$ is: $V_x$ is weakly equivariantly homeomorphic
to an invariant open subset of $\R^n$ under the standard $\Z^n_2$-action defined by: for any $(x_1,\cdots, x_n)\in \R^n$ and  $(g_1,\cdots, g_n)\in \Z^n_2$,
  \[ \quad\ (g_1,\cdots, g_n) \cdot (x_1,\cdots, x_n)\longmapsto \big((-1)^{g_1}x_1,\cdots,
   (-1)^{g_n}x_n \big). \]

   \item Every non-zero element $g\in \Z^n_2$ determines a nontrivial involution $\tau_g$ on $W$, called a \emph{regular involution} on $W$.  
 
 \end{itemize}
\end{defi}

 We will prove in Theorem~\ref{Thm:Locally-Std} that if a  $2$-torus manifold is equivariantly formal, then it must be locally standard.\n
 
 For an $n$-dimensional locally standard $2$-torus manifold $W$, the orbit space $Q=W\slash \Z^n_2$ naturally becomes a 
   connected smooth nice $n$-manifold with corners and with non-empty boundary (since the $\Z^n_2$-action is non-free). Moreover, 
  \begin{itemize}
  \item 
  The $\Z^n_2$-action on $W$ determines a \emph{characteristic function}
   $$\lambda_W: \{F_1,\cdots, F_m\}\rightarrow \Z^n_2$$ where $F_1,\cdots, F_m$ are all the facets (codimension-one faces) of $Q$.\n
  \item The free part of the $\Z^n_2$-action on $W$ determines a principal $\Z^n_2$-bundle $\xi_W$ over $Q$.
  \end{itemize}
  
   It is shown in L\"{u}-Masuda~\cite[Lemma 3.1]{MaLu08}
that $W$ can be recovered from the data $(Q,\lambda_W, \xi_W)$ up to equivariant homeomorphism. 
 In addition, let $\pi: W\rightarrow Q$ denote the orbit map.
    If $f$ is a codimension-$k$ face of $Q$, then $W_f:=\pi^{-1}(f)$ is a codimension-$k$ submanifold of $W$ called a \emph{facial submanifold of $W$}. 
  Let $G_f$ denote the isotropy subgroup of $W_f$.
   Then $W_f$ is also a $2$-torus manifold with respect to the induced action of $\Z^n_2\slash G_f$. 
 In the following, when we say $W_f$ is equivariantly formal, we always consider $W_f$ being equipped with the induced $\Z^n_2\slash G_f$-action from $W$. 
   \n
 
 The main purpose of this paper is to answer the following two questions.\n
  
  \textbf{Question-1:} What kind of
  $2$-torus manifolds are equivariantly formal?\n

   \textbf{Question-2:}  What kind of locally standard $2$-torus manifolds have regular $\mathrm{m}$-involutions?
 \n

 Generally speaking, it is very hard to compute the equivariant cohomology of a locally standard $2$-torus manifold $W$ directly from its orbit space $Q$ and the data $(\lambda_W, \xi_W)$.  So it is difficult to judge whether $W$ is equivariantly formal by directly verifying the condition in the definition. 
 Meanwhile, it was proved by Masuda-Panov~\cite{MasPanov06} that
  a smooth $T^n$-action on a connected smooth
 $2n$-manifold with non-empty fixed points is equivariantly formal if and only if the $T^n$-action is locally standard and every face of its orbit space is acyclic (also see Goertsches-T\"oben~\cite[Theorem 10.19]{GoerTob10} for a reformulation of this result). 
 This result is also implied by Franz~\cite[Theorem 1.3]{Franz17}. The arguments in~\cite{MasPanov06} inspire us to prove the following parallel result for $2$-torus manifolds.
 
   \begin{thm} \label{thm:Main-1}
      Let $W$ be a $2$-torus manifold with orbit space $Q$.
      \begin{itemize}
      \item[(i)] $W$ is equivariantly formal if and only if $W$ is locally standard and $Q$ is mod $2$ face-acyclic.\n
      \item[(ii)] $W$ is equivariantly formal
      and $H^*(W;\Z_2)$
     is generated by its degree-one part as a ring if and only if
      $W$ is locally standard and $Q$ is a mod $2$ homology polytope.
      \end{itemize}
   \end{thm}

 The definitions of ``mod $2$ face-acyclic'' and ``mod $2$ homology polytope'' are given in Definition~\ref{Def:polytope}. 
 \n

    The main strategy in our proof of Theorem~\ref{thm:Main-1} is very similar to the strategy used in~\cite{MasPanov06} for equivariantly formal torus manifolds. Besides, our proof  
  uses the mod $2$ GKM theory introduced in Biss-Guillemin-Holm~\cite{BisGuiHolm04} which allows us to observe the equivariant cohomology of an equivariantly formal
$2$-torus manifold by restricting to its fixed point set (see Section~\ref{Subsec:mod2-GKM}).

 \begin{rem}
  If a $2$-torus manifold $W$ is assumed to be locally standard in the first place, Theorem~\ref{thm:Main-1}\,(i) can also be derived from Chaves~\cite[Theorem 1.1]{Chav20} whose proof uses the theory of syzygies in the mod $2$ equivariant cohomology (see Allday-Franz-Puppe~\cite[Theorem 10.2]{AllFranzPuppe21}) and the mod $2$ ``Atiyah-Bredon sequence''. But we will use a completely different approach in our proof here.
 \end{rem}

 Using Theorem~\ref{thm:Main-1}, we can easily derive the following theorem which gives an answer to Question-$2$.
   
    \begin{thm} \label{thm:Main-2}
      Let $W$ be an $n$-dimensional locally standard $2$-torus manifold with orbit space $Q$. Then there exists a regular $\mathrm{m}$-involution on $W$ if and only if $Q$ is mod $2$ face-acyclic (or equivalently $W$ is equivariantly formal) and the values of
      the characteristic function $\lambda_W$ 
       on all the facets of $Q$ consist exactly of a linear basis of $\Z^n_2$.
   \end{thm}

  A nice manifold with corners $Q$ is called \emph{$k$-colorable} if we can assign $k$ different colors
  to all the facets of $Q$ so that no two adjacent
  facets are of the same color.
 Clearly, there exists a $2$-torus manifold  
      over $Q$ whose characteristic function takes value in a linear basis of $\Z^n_2$ if and only if 
       $Q$ is $n$-colorable.

      \begin{rem}
    By Theorem~\ref{thm:Main-2} and the construction in Puppe~\cite{Puppe01}, we can obtain a self-dual binary code $\mathcal{C}_Q$ from an $n$-colorable mod $2$ face-acyclic nice smooth $n$-manifold with corners $Q$ when $n$ is odd. This generalizes the self-dual binary codes from $n$-colorable simple convex $n$-polytopes in Chen-L\"{u}-Yu~\cite{ChenLuYu18}. Moreover, we can write down $\mathcal{C}_Q$
  explicitly in the same way as the self-dual binary code obtained in~\cite[Corollary 4.5]{ChenLuYu18}. 
   \end{rem}

   The paper is organized as follows. In Section~\ref{Sec-Prelim}, we review the definitions and some basic facts of locally standard
   $2$-torus manifolds and quote some well known results that are useful for our proof.
   In Section~\ref{Sec-Proof-Main-1}, we study various properties of equivariantly formal $2$-torus manifolds. Since the philosophy of our study is very similar to the study
   of torus manifolds with vanishing odd degree cohomology in Masuda-Panov~\cite{MasPanov06}, many lemmas in this paper are parallel to those in~\cite{MasPanov06}.
  In Section~\ref{Sec:Equiv-Formal-Cohom-One}, we prove some special properties of
equivariantly formal $2$-torus manifolds whose mod $2$ cohomology rings are generated by their degree-one part.
 Then finally in Section~\ref{Sec:Proof-Thm-1},  we prove Theorem~\ref{thm:Main-1} and Theorem~\ref{thm:Main-2}.

   \section{Preliminaries} \label{Sec-Prelim}
   
   \subsection{Manifolds with corners and locally standard $2$-torus manifolds} \label{Subsec:Mfd-corners} \ \n
     
  Recall a (smooth) \emph{$n$-dimensional manifold with corners} $Q$ is a
   Hausdorff space together with a maximal
  atlas of local charts onto open subsets of $\R_{\geq 0}^n $
  such that the transition functions are (diffeomorphisms) homeomorphisms which preserve the codimension of each point.
  Here the \emph{codimension} $c(x)$ of a point $x=(x_1,\cdots,x_n)$ in $\R_{\geq 0}^n$ is the number of
  $x_i$ that is $0$. So we have a well defined map
  $c: Q\rightarrow \Z_{\geq 0}$ where $c(q)$ is the codimension of a point $q\in Q$. 
  An \emph{open face} of $Q$ of codimension $k$
  is a connected component of $c^{-1}(k)$. A (closed) 
  \emph{face}
  is the closure of an open face. A face of codimension one is called a \emph{facet} of $Q$. When $Q$ is connected, we also consider $Q$ itself as a face (of codimension zero).
  \begin{itemize}
  \item
  For any $k\in \Z_{\geq 0}$, the \emph{$k$-skeleton} of $Q$ is the union of all the faces of $Q$ with dimension $\leq k$.\n
  \item  
   Th \emph{face poset} of $Q$, denoted by $\mathcal{P}_Q$, is the set of faces of $Q$ ordered by reversed inclusion (so $Q$ is the initial element).
   \end{itemize}
  \n
   
    A manifold with corners $Q$ is said to be \emph{nice} if either its boundary $\partial Q$ is empty or
  $\partial Q$ is non-empty and any codimension-$k$ face of $Q$ is a component of the intersection of
  $k$ different facets in $Q$. If $Q$ is nice,
$\mathcal{P}_Q$ is a simplicial poset. But
in general $\mathcal{P}_Q$
may not be the face poset of a simplicial complex.
Indeed, $\mathcal{P}_Q$ is the face poset of a simplicial complex if and only if all non-empty
multiple intersections of facets of $Q$ are connected (see~\cite[Sec.\,5.2]{MasPanov06}).
\n

\begin{defi} \label{Def:polytope}
 
Let $Q$ be a nice manifold with corners.
 \begin{itemize}
 \item We call $Q$ \emph{mod $2$ face-acyclic} if every face of $Q$ (including $Q$ itself) is a mod $2$ acyclic space. \n

\item We call $Q$ a \emph{mod $2$ homology polytope} if $Q$ is mod $2$ face-acyclic and $\mathcal{P}_Q$
is the face poset of a simplicial complex.
   \end{itemize}
  \end{defi}
  
A topological space $B$ is called \emph{mod $2$ acyclic if $H^*(B;\Z_2)\cong H^*(pt;\Z_2)$.}\n

    It is not difficult to prove the following lemma 
    (see~\cite[p.743\ Remark]{MasPanov06} for a short argument).
    
    \begin{lem} \label{Lem:1-Skeleton-Connect}
    If $Q$ is mod $2$
face-acyclic, then every face of $Q$ has a vertex and the $1$-skeleton of $Q$ is connected.
\end{lem}

  In the following, let $W$ be an $n$-dimensional locally standard $2$-torus manifold with orbit space $Q$. Then $Q$ is a smooth nice manifold with corners with $\partial Q\neq \varnothing$. 
   Let $\pi: W\rightarrow Q$ denote the projection and let the set of facets of $Q$ be
  $$\mathcal{F}(Q) =\{F_1,\cdots, F_m \}.$$ 
  
    Then
   $\pi^{-1}(F_1),\cdots,\pi^{-1}(F_m)$ are embedded codimension-one closed connected submanifolds of $W$, called the \emph{characteristic submanifolds} of $W$. Moreover,
   the $\Z^n_2$-action on $W$ determines a \emph{characteristic function} on $Q$ which is a map
   \begin{equation} \label{Equ:Char-Func}
    \lambda_W: \mathcal{F}(Q)\rightarrow \Z^n_2 
    \end{equation}  
    where $\lambda_W(F_i)\in \Z^n_2$ is the generator of the $\Z_2$ subgroup that pointwise fixes the submanifold $\pi^{-1}(F_i)$, $1\leq i \leq m$. Since the $\Z^n_2$-action is locally standard, the function
 $\lambda_W$ satisfies the following \emph{linear independence condition}: 
     \begin{equation*} 
  \quad \begin{array}{l}
        \text{whenever the intersection of $k$ different facets $F_{i_1},\cdots, F_{i_k}$ is non-empty,} \\
      \text{the elements $\lambda_W(F_{i_1}),\cdots , \lambda_W(F_{i_k})$ are linearly independent when viewed} \\
     \text{as vectors of $\Z^n_2$ over the field $\Z_2$.}    \end{array}
   \end{equation*}
   
   For a codimension-$k$ face $f$ of $Q$, let
 $F_{i_1},\cdots, F_{i_k}$ be all the facets containing $f$. Then the isotropy subgroup of the facial submanifold $W_f$ is
   \begin{equation} \label{Equ:Isotropy-Sub-Mfd}
      G_f = \text{the subgroup generated by $\{\lambda_W(F_{i_1}),\cdots , \lambda_W(F_{i_k})\}$} \subseteq \Z^n_2.
      \end{equation} \n 
   
   By the linear independence condition of $\lambda_W$, $G_f\cong \Z^k_2$. Hence $W_f$ is also a $2$-torus manifold with respect to the induced action of $\Z^n_2\slash G_f \cong \Z^{n-k}_2$.\n
   
    In addition, $W$ determines a \emph{principal $\Z^n_2$-bundle} over $Q$ as follows. We take a small invariant open tubular
neighborhood for each characteristic submanifold of $W$ and remove their
union from $W$. Then the $\Z^n_2$-action on the resulting space is free and
its orbit space can naturally be identified with $Q$, which gives a principal $\Z^n_2$-bundle over $Q$, denoted by $\xi_W$. It is shown in L\"{u}-Masuda~\cite{MaLu08} that
$W$ can be recovered (up to equivariant homeomorphism) from ($Q,
\xi_W,\lambda_W$). For example,
  when $\xi_W$ is a trivial $\Z^n_2$-bundle, $W$ is equivariantly homeomorphic to the following ``\emph{canonical model}'' determined by $(Q,\lambda_W)$.
    \begin{equation} \label{Equ-Glue-Const-Z2}
      M_Q(\lambda_W) := Q\times \Z^n_2 \slash \sim  
     \end{equation} 
    where $(q,g)\sim (q',g')$ if and only if $q=q'$ and
    $g-g'\in G_{f(q)}$ where $f(q)$ is the unique face
    of $Q$ that contains $q$ in its relative interior. 
   This canonical model is a generalization of a result of Davis-Januszkiewicz~\cite[Prop.\,1.8]{DaJan91}.
   We will see that the canonical model plays an important role in our proof of Theorem~\ref{thm:Main-1} in Section~\ref{Sec:Proof-Thm-1}.  
    
     \subsection{Borel construction and equivariant cohomology}
\ \n
  
  For a topological group $G$, there exists
a contractible free right $G$-space $EG$ called the \emph{universal $G$-space}. The quotient
$BG= EG \slash G$ is called the \emph{classifying space} for free $G$-actions.  For example when $G=\Z_2^n$, we can choose
 $$E\Z^n_2 = (E\Z_2)^n = (S^{\infty})^n, \ \ B\Z_2^n = (B\Z_2)^n = (\R P^{\infty})^n.$$
  
  Let $X$ be a topological space with a left $G$-action (we call $X$ a \emph{$G$-space} for brevity). The \emph{Borel construction} of $X$ is denoted by
  $$  EG\times_G X = 
     EG \times X  \slash \sim $$
    where $(e,x)\sim (eg,g^{-1}x)$ for any $e\in 
    EG$, $x\in X$ and $g\in G$.\n

The \emph{equivariant cohomology} of $X$ with coefficients in a field $\mathbf{k}$ is defined as
   $$H^*_G(X;\mathbf{k}) := H^*(EG\times_G X;\mathbf{k}) .$$
  
\noindent  \textbf{Convention:}  
The term ``cohomology'' of a space in this paper,  always mean singular cohomology if not specified otherwise.\n

 The Borel construction determines a canonical fibration called \emph{Borel fibration}:
    \begin{equation} \label{Equ:Borel-Fiber-Bundle}
       X \rightarrow EG\times_G X \rightarrow BG. 
     \end{equation}

 The map $\rho$ collapsing $X$ to a point induces a homomorphism
 \begin{equation} \label{Equ:rho-star}
   \rho^* : H^*_G(pt ;\mathbf{k}) = H^*(BG;\mathbf{k})
\rightarrow H^*_G(X;\mathbf{k})  
\end{equation}
 which defines a canonical $H^*(BG;\mathbf{k})$-module structure on $H^*_G(X;\mathbf{k})$. A useful fact is:
 when $X$ is a paracompact space with finite cohomology dimension, and $G=T^r$ or $(\Z_p)^r$ where $p$ is a prime
and $\mathbf{k}=\Q$ or $\Z_p$ respectively, $\rho^*$ is injective if and only if the fixed point set $X^G$ is non-empty (see~\cite[Ch.\,IV]{Hsiang75}). \n

In general, $H^*_G(X;\mathbf{k})$ may not be a free $H^*(BG;\mathbf{k})$-module. The following \emph{localization theorem} due to A.~Borel (see~\cite[p.\,45]{Hsiang75}) says that we can compute the free $H^*(BG;\mathbf{k})$-module part of $H^*_G(X;\mathbf{k})$ by restricting to  the fixed point set. 

\begin{thm}[Localization Theorem] \label{Thm:Local-Borel}
  Let $G=T^r$ or $(\Z_p)^r$ where $p$ is a prime
and $\mathbf{k}=\Q$ or $\Z_p$ respectively. For a paracompact
$G$-space $X$ with finite cohomology dimension, the following localized restriction homomorphism is an isomorphism:
\[  S^{-1} H^*_G(X;\mathbf{k}) \rightarrow S^{-1} H^*_G(X^G;\mathbf{k}) = H^*(X^G;\mathbf{k})\otimes_{\mathbf{k}} (S^{-1} H^*(BG;\mathbf{k})) \]
where $S=R-\{0\}$ and $R$ is the polynomial subring of $H^*(BG;\mathbf{k})$. So the kernel of the restriction
$H^*_G(X;\mathbf{k}) \rightarrow  H^*_G(X^G;\mathbf{k})$
lies in the $H^*(BG;\mathbf{k})$-torsion of
$H^*_G(X;\mathbf{k})$. In particular if $X$ is 
equivariantly formal, $H^*_G(X;\mathbf{k})\rightarrow  H^*_G(X^G;\mathbf{k})$ is injective.
\end{thm}

  The Borel construction
can also be applied to a $G$-vector bundle
$\pi: E\rightarrow X$ (i.e. both $E$ and $X$ are 
$G$-spaces and the projection $\pi$ is $G$-equivariant).
 In this case, the Borel construction $E_G$ of $E$ is a
vector bundle over $X_G$ whose mod $2$ Euler class, denoted by $e^G(E)$, lies in $H^*_G(X;\Z_2)$.
Note that using $\Z_2$-coefficients allows us to ignore the orientation of a vector bundle.

    \subsection{Mod $2$ GKM-theory} \label{Subsec:mod2-GKM}
    \ \n
    
 Let $W$ be an $n$-dimensional equivariantly formal $2$-torus manifold. Then the fixed point set $W^{\Z^n_2}$ is
 a finite non-empty set (by Theorem~\ref{thm:AllPuppe}), and $H^*_{\Z^n_2}(W;\Z_2)$ is a free module over $H^*(B\Z^n_2;\Z_2)$. Moreover, $H^*_{\Z^n_2}(W;\Z_2)$ can be computed 
    by the so called Mod $2$ GKM-theory (see 
    Biss-Guillemin-Holm~\cite{BisGuiHolm04}) which is an extension of the GKM-theory in~\cite{GKM98} to $2$-torus actions. In this section, we briefly review some results related to our study. The reader is referred to~\cite{BisGuiHolm04} and~\cite{Lu05} for more details.\n
    
    For each $1\leq i \leq n$, let $\rho_i\in
\mathrm{Hom}(\Z^n_2,\Z_2)$ be the homomorphism defined by
$$ \rho_i((g_1,\cdots, g_n))=g_i,\ \forall 
(g_1,\cdots, g_n)\in \Z^n_2. $$ 

 By a canonical isomorphism $\mathrm{Hom}(\Z^n_2,\Z_2)\cong H^1(B\Z^n_2;\Z_2)$,
 we can identify $H^*(B\Z^n_2;\Z_2)$ with 
  the graded polynomial ring $\Z_2[\rho_1,\cdots,\rho_n]$
  where $\mathrm{deg}(\rho_i)=1$, $1\leq i \leq n$.\n
 
 Let $Q=W\slash \Z^n_2$ be the orbit space of $W$.
 By our Theorem~\ref{Thm:Locally-Std} proved later, a $2$-torus manifold $W$ being equivariantly formal implies that it is locally standard. Hence $Q$ is a nice manifold with corners. Then the $1$-skeleton of $Q$, consisting of vertices
(0-faces) and edges (1-faces) of $Q$, is an $n$-valent graph denoted by $\Gamma(Q)$. Let $V(Q)$ and $E(Q)$ denote the set of vertices and edges of $Q$, respectively.\n

\noindent \textbf{Convention:}
We will not distinguish a vertex of $Q$
and the corresponding fixed point in $W^{\Z^n_2}$
in the rest of the paper. \n

 \begin{itemize}
 \item Let $\pi: W\rightarrow Q$ be the quotient map.\n
 \item
  For each edge $e\in E(Q)$, $\pi^{-1}(e)$ is a circle whose isotropy subgroup $G_e$ is a rank $n-1$ subgroup of $\Z^n_2$.
Then we obtain a map
\[ \alpha: E(Q) \rightarrow \mathrm{Hom}(\Z^n_2,\Z_2)\cong H^1(B\Z^n_2;\Z_2)  \]
   where for each edge $e\in E(Q)$, $\ker(\alpha(e))=G_e$.
     \n
  \item 
   For each vertex $p\in V(Q)$, let $\alpha_p = \{ \alpha(e)\,|\, p\in e \}\subset \mathrm{Hom}(\Z^n_2,\Z_2) $.\n
 \end{itemize}

 Such a map $\alpha$ is called an \emph{axial function} which has the following properties:
   \begin{itemize}
     \item[(i)] For every vertex $p\in V(Q)$, $\alpha_p$ is a linear basis of $\mathrm{Hom}(\Z^n_2,\Z_2)$.\n
     
     \item[(ii)] For every edge $e\in E(Q)$, $\alpha_p \equiv \alpha_{p'}$ mod $\alpha(e)$ where $p,p'$ are the two vertices of $e$.
   \end{itemize}

   By~\cite[Theorem C]{BisGuiHolm04} and~\cite[Remark 5.9]{BisGuiHolm04}, we have the following theorem which is a consequence of the $\Z_2$-version Chang–Skjelbred
 theorem (see~\cite[Theorem 4.1]{BisGuiHolm04} 
 and~\cite{ChangSkj74}).
 
    \begin{thm}[see~\cite{BisGuiHolm04}]\label{thm:compute}
    Let $W$ be an $n$-dimensional equivariantly formal  $2$-torus manifold.
    If we choose an element $\eta_p \in H^*_{\Z^n_2}(W^{\Z^n_2};\Z_2)$ for each $p\in W^{\Z^n_2}$,
    then 
    $$(\eta_p)\in \bigoplus_{p\in W^{\Z^n_2}} 
    H^*(B\Z^n_2;\Z_2) \cong H^*_{\Z^n_2}(W^{\Z^n_2};\Z_2)$$
     is in the image of the restriction homomorphism
   $r: H^*_{\Z^n_2}(W;\Z_2)  \rightarrow H^*_{\Z^n_2}(W^{\Z^n_2};\Z_2)$ if and only if
   for every edge $e\in E(Q)$ with vertices $p$ and $p'$,
   $\eta_p - \eta_{p'}$ is divisible by $\alpha(e)$.
    \end{thm}
    
  Moreover, we can understand the above axial function $\alpha$ in the following way. 
   For brevity, we use the following notations for an $n$-dimensional locally standard $2$-torus manifold $W$
   in the rest of this section. 
  
   \begin{itemize}
    \item Let $G=\Z^n_2$.\n
    
   \item Let $W_i:=W_{F_i}=\pi^{-1}(F_i)$, $1\leq i \leq m$, be all the characteristic submanifolds of $W$ where $F_1,\cdots, F_m$ are all the facets of $Q$.\n
   
  \item Let $G_i:=\langle\lambda_W(F_i)\rangle\cong \Z_2$ be subgroup of $G$ that fixes $W_i$ pointwise.
  \n
  
   \item Let $\nu_i$ be 
  the (equivariant) normal bundle of $W_i$ in $W$. So we have the \emph{equivariant Euler class} of $\nu_i$, denoted by $e^G(\nu_i)\in H^1_G(W_i;\Z_2)$. \n
  
  \item For any fixed point $p\in W^{\Z^n_2}$, let 
  $ I(p):=\{ i \,|\, p\in W_i \}$.  We have the decomposition of tangent space $T_p W$ as
  \[  T_p W = \bigoplus_{i\in I(p)} \nu_i|_p. \]
   where $\nu_i|_p$ denotes the restriction of $\nu_i$ to $p$. So $\nu_i|_p$ is a $1$-dimensional linear representation of $G$ whose
  equivariant Euler class 
  $$e^G(\nu_i|_p) = e^G(\nu_i)|_p\in H^1(B\Z^n_2;\Z_2).$$ \n
 \end{itemize}

   The inclusion map $\psi_i: W_i\hookrightarrow W$ defines an equivariant Gysin homomorphism $\psi_{i_{!}}: H^*_{G}(W_i;\Z_2)\rightarrow
  H^{*+1}_{G}(W;\Z_2)$ (see~\cite[\S 5.3]{AllPuppe93} for example). For brevity, let 
  $$\tau_i =\tau_{F_i}= \psi_{i_{!}}(1)\in H^1_{G}(W;\Z_2)$$
   be the image of the identity $1 \in H^0_{G}(W_i;\Z_2)$. The element $\tau_i$ can be thought of as the Poincar\'e dual of the Borel construction of $W_i$ in $H^*_{G}(W;\Z_2)$ and is called the \emph{equivariant Thom class} of $\nu_i$. A standard fact is
  \[ \text{$\tau_i|_p$ agrees with the equivariant Euler 
  class of $\nu_i|_p$}. \]

  Note that the elements of $\mathrm{Hom}(\Z^n_2,\Z_2)$ are in one-to-one correspondence with all
  the $1$-dimensional linear representations of $\Z^n_2$.
  So the canonical
isomorphism between $\mathrm{Hom}(\Z^n_2,\Z_2)$ and
$H^1(B\Z^n_2;\Z_2)$ is given by the equivariant Euler class of a $1$-dimensional representations of $\Z^n_2$. 
Then we have the following identification:
\begin{equation} \label{Equ:Weight-Identify}
\alpha_p = \{ \alpha(e)\,|\, p\in e \} \longleftrightarrow \{  e^G(\nu_i)|_p=\tau_i|_p\,;\, i\in I(p) \}.
\end{equation}
 where an edge $e$ containing $p$ corresponds to the unique index $i\in I(p)$ so that the facet $F_i$ intersects $e$ transversely (or equivalently $e \nsubseteq F_i$).\n
 
  \begin{itemize}
   \item For a codimension-$k$ face $f$ of $Q$,
  let $\nu_f$ denote the (equivariant) normal bundle of $W_f$ in $W$. Denote by $\tau_f \in H^k_G(W;\Z_2)$ the equivariant Thom class of $\nu_f$. Then the restriction of $\tau_f$ to
  $H^k_G(W_f;\Z_2)$ is the equivariant Euler class of $\nu_f$, denoted by $e^G(\nu_f)$. In particular, if $f=Q$, $W_f=W$ and so $\tau_f$ is the identity element of $H^0_G(W_f;\Z_2)$. 
  \end{itemize}
 \n
 
   Let $r_p: H^*_G(W;\Z_2)\rightarrow H_G^*(p;\Z_2)\cong H^*(BG;\Z_2)$ denote the restriction map at a fixed point $p\in W^G$.  Then 
  \begin{equation} \label{Equ:Local-map}
    r=\bigoplus_{p\in W^G} r_p : H^*_G(W;\Z_2)\rightarrow
     H^*_G(W^G;\Z_2) = \bigoplus_{p\in W^G} H^*(BG;\Z_2). 
    \end{equation}  
  By Theorem~\ref{Thm:Local-Borel}, the kernel of $r$ is the $H^*(BG;\Z_2)$-torsion subgroup of $H^*_G(W;\Z_2)$.\n

  Clearly, $r_p(\tau_f) =0$ unless $p\in (W_f)^G$ (i.e. $p$ is a vertex of $f$).
  It follows from~\eqref{Equ:Weight-Identify} that
  for any $p\in W^G$,
  \begin{equation} \label{Equ:res-tau}
  r_p(\tau_f)=   \begin{cases}
   \underset{p\in e,\, e\nsubseteq f}{\prod} \alpha(e),  &  \text{if  $p \in f$}; \\
  \ \ 0,  &  \text{otherwise}.
 \end{cases} 
  \end{equation}
  
 In addition, define
    \begin{equation} \label{Equ:wideHat-H}
      \widehat{H}^*_G(W;\Z_2):= H^*_G(W;\Z_2) \big\slash H^*(BG;\Z_2)\text{-torsion}.
    \end{equation}
  
  By the localization theorem (Theorem~\ref{Thm:Local-Borel}), the restriction homomorphism $r$ induces a monomorphism 
  $ \widehat{H}^*_G(W;\Z_2) \rightarrow H^*_G(W^G;\Z_2)$, still denoted by $r$.\nn
  
    The following proposition is parallel to~\cite[Proposition 3.3]{MasPanov06}.
    \begin{prop} \label{Prop:Equiv-Struc}
      Let $W$ be an $n$-dimensional locally standard $2$-torus manifold.
      \begin{itemize}
      \item[(i)] For each characteristic submanifold $W_i$ with
           $(W_i)^{G}\neq \varnothing$ where $G=\Z^n_2$, there is a unique element $a_i\in H_1(BG;\Z_2)$ such that
            $$\rho^*(t) = \sum_{i}\langle t,a_i\rangle \tau_i \ \, \text{modulo}\ \, H^*(BG;\Z_2)\text{-torsion} $$
            for any element $t\in H^1(BG;\Z_2)$. Here the sum is taken over all the characteristic submanifolds $W_i$ with $(W_i)^G\neq \varnothing$ and $\rho^*$ is defined in~\eqref{Equ:rho-star}.\n
      \item[(ii)] For each $W_i$ with $(W_i)^G\neq \varnothing$, the subgroup $G_i$ fixing $W_i$ coincides with the subgroup determined by $a_i\in H_1(BG;\Z_2)$ through the identification 
      $H_1(BG;\Z_2)\cong \mathrm{Hom}(\Z_2,G)$.\n
      
      \item[(iii)] If $n$ different characteristic submanifolds
      $W_{i_1},\cdots, W_{i_n}$ have a $G$-fixed point
      in their intersection, then the elements
      $a_{i_1},\cdots,a_{i_n}$ form a linear basis of 
      $H_1(BG;\Z_2)$ over $\Z_2$.
   \end{itemize}
 \end{prop}
    \begin{proof}
     The argument is completely parallel to the arguments
     for torus manifolds in the proof of~\cite[Lemma 1,3, Lemma. 1.5, Lemma 1.7]{Mas99}. Indeed, we can just replace
    the torus manifold $M$ in~\cite{Mas99} by our $2$-torus manifold $W$ and, replace $T^n$ by $\Z^n_2$ and $H^2(M ;\Z)$ by $H^1(W ;\Z_2)$ to obtain our proof
    here. The details of the proof are left to the reader.
    \end{proof} 
    
    In addition, the following lemma is completely parallel to the torus manifold case~\cite[Lemma 6.2]{MasPanov06}. Its proof is also parallel to~\cite{MasPanov06}, hence omitted.\n
    
    \begin{lem} \label{Lem:Diff-Restrict}
      Let $W$ be a locally standard $2$-torus manifold with orbit space $Q$.
    For any $\eta\in H^*_G(W;\Z_2)$ and any edge $e\in E(Q)$, $r_p(\eta)-r_{p'}(\eta)$ is divisible by
    $\alpha(e)$ where $p$ and $p'$ are the endpoints of $e$.
    \end{lem}

    \subsection{Face ring}
   \ \n
   
    A poset (partially ordered set) $\mathcal{P}$ is called \emph{simplicial} if it has an initial element
    $\hat{0}$ and for each $x\in \mathcal{P}$ the lower segment $[\hat{0},x]$ is a boolean lattice (the face lattice of a simplex). \n
    
   Let $\mathcal{P}$ be a simplicial poset. For each $x\in \overline{\mathcal{P}}:=\mathcal{P}-\{\hat{0}\}$, we assign a geometrical simplex whose face poset 
    is $[\hat{0},x]$ and glue these geometrical simplices together according to the order relation in $\mathcal{P}$.
    The cell complex we obtained is called the \emph{geometrical realization} of $\mathcal{P}$, denoted by $|\mathcal{P}|$. We may also say that $|\mathcal{P}|$
    is a \emph{simplicial cell complex}.\n
    
    For any two elements $x, x'\in \mathcal{P}$, 
    denote by $x\vee x'$ the set of their least common
    upper bounds, and by $x\wedge x'$ their greatest common lower bounds. Since $\mathcal{P}$ is simplicial,
    $x\wedge x'$ consists of a single element if 
    $x\vee x'$ is non-empty.
    
    \begin{defi}[see Stanley~\cite{Stanley91}]
       The \emph{face ring} of a simplicial poset $\mathcal{P}$ over a field $\mathbf{k}$ is the quotient
       $$\mathbf{k}[\mathcal{P}] := \mathbf{k}[v_x:x\in\mathcal{P}] \big\slash \mathcal{I}_{\mathcal{P}}$$
       where $\mathcal{I}_{\mathcal{P}}$ is the ideal generated by all the elements of the form
         $$ v_x v_{x'} - v_{x\wedge x'}\cdot \sum_{x''\in x\vee x'} v_{x''}. $$
     \end{defi}  
   
    Let $Q$ be a nice manifold with corners.
    It is easy to see that the face poset of $Q$ is a
    simplicial poset, denoted by $\mathcal{P}_Q$.
    We call $|\mathcal{P}_Q|$ the simplicial cell complex \emph{dual} to $Q$.\n
    
    We define the \emph{face ring} of $Q$ to be 
    the face ring of $\mathcal{P}_Q$. Equivalently,    
        we can write the face ring of $Q$ as
        \[ \mathbf{k}[Q]:=\mathbf{k}[v_f: f \ 
        \text{a face of $Q$}] \big\slash \mathcal{I}_Q.  \]
        where $\mathcal{I}_Q$ is the ideal generated by all the elements of the form
        \[  v_f v_{f'} - v_{f\vee f'}\cdot \sum_{f''\in f\cap f'} v_{f''}.\]
        where $f\vee f'$ denotes the unique minimal face
        of $Q$
    containing both $f$ and $f'$. \n
    
   \noindent \textbf{Convention:}        
      For any face $f$ of $Q$, define the degree of $v_f$ to be the codimension of $f$. Then $\mathbf{k}[Q]=\mathbf{k}[\mathcal{P}_Q]$ becomes a graded ring. Note that in the discussion of torus manifolds in~\cite{MasPanov06}, the degree of $v_f$ is defined to be twice the codimension of $f$ to fit the study there.
         \nn
         
         The \emph{$f$-vector} of $Q$ is defined as
         $\mathbf{f}(Q) =(f_0,\cdots, f_{n-1})$
         where $n=\dim(Q)$ and $f_i$ is the number of faces of codimension $i+1$. The equivalent
information is contained in the \emph{$h$-vector} $\mathbf{h}(Q)=(h_0,\cdots,h_n)$ determined by the
equation:
\begin{equation} \label{Equ:h-vec}
   h_0 t^n +\cdots + h_{n-1}t+ h_n = (t-1)^n + f_0 (t-1)^{n-1}+\cdots+ f_{n-1}.
 \end{equation}
       
     The \emph{Hilbert series} of $\mathbf{k}[Q]$ is $F(\mathbf{k}[Q]):=\sum_i \dim_{\mathbf{k}}\mathbf{k}[Q]_{i} \cdot t^i$ where $\mathbf{k}[Q]_{i}$ denotes the homogeneous degree $i$ part of $\mathbf{k}[Q]$. 
       By~\cite[Proposition 3.8]{Stanley91},
       \begin{equation} \label{Equ:Hilbert-Face-Ring}
         F(\mathbf{k}[Q];t) = \frac{h_0 + h_1 t +\cdots + h_nt^n}{(1-t)^n}.
       \end{equation}  
         
          The following construction is taken from~\cite[Section 5]{MasPanov06}. For any vertex ($0$-face) $p\in Q$, we define a map
   \begin{equation} \label{Equ:s-p}
    s_p: \mathbf{k}[Q] \rightarrow \mathbf{k}[Q]\big\slash
    (v_f: p\notin f).
    \end{equation} 
    
   If $p$ is the intersection of $n$ different facets
   $F_1,\cdots, F_n$, then $\mathbf{k}[Q]\big\slash
    (v_f: p\notin f)$ can be identified with the polynomial ring $\mathbf{k}[v_{F_1},\cdots,v_{F_n}]$.  
    \n
    
    \begin{lem}[Lemma 5.6 in~\cite{MasPanov06}]\label{Lem:res-s}
     If every face of $Q$ has a vertex, then the
    direct sum $s=\bigoplus_{p} s_p$ over all vertices
    $p\in Q$ is a monomorphism from
    $\mathbf{k}[Q]$ to the sum of polynomial rings
    $\mathbf{k}[Q]\big\slash
    (v_f: p\notin f)$. 
    \end{lem}

          A finitely generated graded commutative ring $R$ over $\mathbf{k}$ is called \emph{Cohen-Macaulay} if there exists an \emph{h.s.o.p} (homogeneous system of parameters) $\theta_1,\cdots,\theta_n$ such that $R$ is a free $\mathbf{k}[\theta_1,\cdots,\theta_n]$-module.
  Clearly, if $\mathbf{k}[Q]=\mathbf{k}[\mathcal{P}_Q]$ is Cohen-Macaulay,
  then it has a \emph{l.s.o.p} (linear system of parameters). \n
  
  A simplicial complex $K$ is called a \emph{Gorenstein* complex} over $\mathbf{k}$ if its face ring $\mathbf{k}[K]$ is Cohen-Macaulay and $H^*(K;\mathbf{k})\cong H^*(S^{d};\mathbf{k})$ where $d=\dim(K)$. The reader is referred to Bruns-Herzog~\cite{BruHerz98} and Stanley~\cite{Stanley07}  for more information of Cohen-Macaulay rings and Gorenstein* complexes.\n
        
      The following proposition is parallel to~\cite[Lemma 8.2\,(1)]{MasPanov06}.  
          
      \begin{prop} \label{Prop:Hom-Polytope}
        If $Q$ is an $n$-dimensional mod $2$ homology polytope, then the geometrical realization $|\mathcal{P}_Q|$ of $\mathcal{P}_Q$ is a Gorenstein* simplicial  complex over $\Z_2$. In particular, $\Z_2[\mathcal{P}_Q]$ is Cohen-Macaulay and $H^*(|\mathcal{P}_Q|;\Z_2)\cong H^*(S^{n-1};\Z_2)$. 
      \end{prop} 
      \begin{proof}
        The proof is almost identical to the proof in~\cite[Lemma 8.2]{MasPanov06} except that we use $\Z_2$-coefficients instead of $\Z$-coefficients
        when applying~\cite[II 5.1]{Stanley07} in the argument.        
      \end{proof}

     \section{Equivariantly formal $2$-torus manifolds} \label{Sec-Proof-Main-1}
     
     In this section, we study various properties of  equivariantly formal $2$-torus manifolds. 
    One may find that many discussions on $2$-torus manifolds here are parallel to the discussions in~\cite{MasPanov06} on torus manifolds. The condition ``vanishing of odd degree cohomology'' on a torus manifold in~\cite{MasPanov06} is now replaced by the equivariant formality condition on a $2$-torus manifold and, the coefficients $\Z$ is replaced by $\Z_2$. Many arguments in~\cite{MasPanov06} are
     transplanted into our proof here
     while some of them actually become simpler.
 \n
      
      In Section~\ref{Subsec:Form-Local-Std}, we prove
      some general results of equivariantly formal $\Z^r_2$-actions on compact manifolds.
      In particular, we prove that any equivariantly formal
     $2$-torus manifold is locally standard, and the equivariant formality of a $2$-torus manifold is inherited by all its facial submanifolds.\n
     
     In Section~\ref{Subsec:Equiv-Local-Std-Mfd}, we explore the relations between
      the equivariant cohomology of a locally standard $2$-torus manifold and the face ring of its orbit space.\n
      
    In Section~\ref{Subsec:Blow-up}, we prove that the equivariant formality of a $2$-torus manifold is preserved under real
     blow-ups along its facial submanifolds. Our proof uses a result from Gitler~\cite{Gilter92}. 
     
         \subsection{Equivariantly formal $\Rightarrow$ locally standard} \label{Subsec:Form-Local-Std}
     
   \begin{lem} \label{Lem:Component-Action}
   Suppose $M$ is a compact manifold whose connected components are
   $M_1,\cdots, M_k$.  A $\Z_2^r$-action on $M$ is equivariantly formal if and only if each $M_i$ is $\Z_2^r$-invariant and the restricted $\Z_2^r$-action on $M_i$ is equivariantly formal.
   \end{lem}
   \begin{proof}
    The ``if'' part is obvious. For the ``only if'' part, assume that $M_{1},\cdots, M_{s}$, $ s\leq k$, are all the components each of which is
   preserved under the $\Z_2^r$-action. 
  Since the $\Z_2^r$-action on $M$ is equivariantly formal, by Theorem~\ref{thm:AllPuppe} we have
   $$\dim_{\Z_2}H^*(M^{\Z^r_2};\Z_2)=
  \dim_{\Z_2}H^*(M;\Z_2).$$
  So in particular, $M^{\Z^r_2}$ is not empty.
  Clearly, $M^{\Z^r_2}$ must lie in
   $M_{1}\cup\cdots\cup M_{s}$, so $s>0$ and $M^{\Z^r_2}$ is the disjoint union of $M_1^{\Z^r_2},\cdots, M_s^{\Z^r_2}$. Then by Theorem~\ref{thm:AllPuppe}, 
  $$ \dim_{\Z_2}H^*(M^{\Z^r_2};\Z_2) = \sum^s_{i=1} \dim_{\Z_2}H^*(M_i^{\Z^r_2};\Z_2) \leq \sum^s_{i=1} \dim_{\Z_2}H^*(M_i;\Z_2)\, \leq \dim_{\Z_2}H^*(M;\Z_2). $$
  
By comparing this inequality with the previous equation, we
can deduce that $s=k$ and on every component $M_i$, $\dim_{\Z_2}H^*(M_i^{\Z^r_2};\Z_2) = \dim_{\Z_2}H^*(M_i;\Z_2)$. So by Theorem~\ref{thm:AllPuppe} again, the 
  $\Z_2^r$-action on $M_i$ is equivariantly formal.    
   \end{proof}
    
    \begin{lem} \label{Lem:Induced-Formality}
   If a $\Z_2^r$-action on a compact manifold $M$ is
    equivariantly formal, then for every subgroup $H$ of $\Z_2^r$,
   \begin{itemize}
    \item[(i)] The action of $H$ on $M$ is 
  equivariantly formal. \n
 \item[(ii)] The induced action of $\Z^r_2$ on $M^H$ and   $\Z^r_2\slash H$ on $M^H$ are both equivariantly formal.\n
 \item[(iii)] The induced action of $\Z^r_2$ (or $\Z^r_2\slash H$) on every connected component $N$ of
 $M^H$ is equivariantly formal, hence $N$ has a $\Z^r_2$-fixed point.  
  \end{itemize}
 \end{lem}
 \begin{proof}
 (i)  By Theorem~\ref{thm:AllPuppe}, it is equivalent to prove
 \begin{equation} \label{Equ:M-H-Fix}
   \dim_{\Z_2} H^*(M^H;\Z_2)=\dim_{\Z_2} H^*(M;\Z_2).
  \end{equation} 
  
   Otherwise assume
 $\dim_{\Z_2} H^*(M^H;\Z_2)<\dim_{\Z_2} H^*(M;\Z_2)$. Observe that the $\Z^r_2$-action on $M$ induces an action of $\Z_2^r\slash H$ on
 $M^H$ and we have 
 \begin{equation} \label{Equ:FixPt-Quotient}
    M^{\Z^r_2} = (M^H)^{\Z^r_2 \slash H}.
    \end{equation}
   So by Theorem~\ref{thm:AllPuppe}, $\dim_{\Z_2} H^*(M^{\Z^r_2};\Z_2)\leq \dim_{\Z_2} H^*(M^{H};\Z_2)
  <\dim_{\Z_2} H^*(M;\Z_2)$, which contradicts the assumption that the $\Z_2^r$-action on $M$ is equivariantly formal. This proves (i).\n
  
  (ii)
  By~\eqref{Equ:FixPt-Quotient} and the assumption that the $\Z^r_2$-action is equivariantly formal, 
  \begin{align*}
    \dim_{\Z_2} H^*\big((M^H)^{\Z^r_2 \slash H};\Z_2\big)
  & = \dim_{\Z_2} H^*\big( M^{\Z^r_2};\Z_2\big) \\
  &= \dim_{\Z_2} H^*\big( M;\Z_2\big)\overset{\eqref{Equ:M-H-Fix}}{=} \dim_{\Z_2} H^*(M^H;\Z_2).
  \end{align*}
   Then by Theorem~\ref{thm:AllPuppe}, the
   action of $\Z^r_2\slash H$ on $M^H$ is equivariantly formal, so is the action of $\Z^r_2$ on $M^H$.\n
   
   (iii) By the conclusion in (ii) and Lemma~\ref{Lem:Component-Action}, the induced action of $\Z^r_2$ (or $\Z^r_2\slash H$)  on every connected component $N$ of $M^H$ is equivariantly formal. So by Theorem~\ref{thm:AllPuppe}, $N$ must have a $\Z^r_2$-fixed point.
 \end{proof} 
  
  Next, we prove a theorem that is parallel to~\cite[Theorem 4.1]{MasPanov06}.
 
 \begin{thm} \label{Thm:Locally-Std}
   If a $2$-torus manifold $W$ is equivariantly formal,
   then $W$ must be locally standard.
 \end{thm}
 \begin{proof}
 Suppose $\dim(W)=n$. For a point $x\in W$, denote by $G_x$ the isotropy group of $x$. 
 \begin{itemize}
 \item If $G_x$ is trivial, then $x$ is in a free orbit
  of the $\Z_2^n$-action. So $W$ is locally standard near $x$.
  \n
 
 \item Otherwise, let $N$ be the connected component of
  $W^{G_x}$ containing $x$. By Lemma~\ref{Lem:Induced-Formality}\,(iii), the induced $\Z_2^n$-action on $N$ has a fixed point, say $x_0$. Since $W^{\Z^n_2}$ is discrete,
    the tangential $\Z^n_2$-representation
$T_{x_0}W$ is faithful.  Then  
  since $x$ and $x_0$
   are in the same connected component fixed
    pointwise by $G_x$, the $G_x$-representation on 
    $T_x W$ agrees with the restriction
of the tangential $\Z^n_2$-representation 
   $T_{x_0} W$ to $G_x$.
    This implies that $W$ is locally
standard near $x$. 
\end{itemize}
 The theorem is proved.
 \end{proof}

 \begin{prop} \label{Prop:Weak-Facial-2-torus}
    Let $W$ be an equivariantly formal $2$-torus manifold  
   with orbit space $Q$.  For any face $f$ of $Q$,
   the facial submanifold $W_f$ is also an 
  equivariantly formal $2$-torus manifold.
 \end{prop}
 \begin{proof}
   Suppose $\dim(W)=n$ and $f$ is a codimension-$k$ face of $Q$.  By Theorem~\ref{Thm:Locally-Std}, $W$ is locally standard. Then $W_f$ is a connected $(n-k)$-dimensional embedded submanifold of $W$ fixed pointwise by $G_f\cong \Z^k_2$ (see~\eqref{Equ:Isotropy-Sub-Mfd}). By Lemma~\ref{Lem:Induced-Formality}\,(iii), the induced action of $\Z^n_2\slash G_f \cong \Z^{n-k}_2$ on $W_f$ is
 equivariantly formal. 
 \end{proof}

 \subsection{Equivariant cohomology of locally standard $2$-torus manifolds} \label{Subsec:Equiv-Local-Std-Mfd}
 \ \n
  Let $W$ be an $n$-dimensional locally standard $2$-torus manifold with orbit space $Q$.
  We explore the relation between
  $H^*_G(W;\Z_2)$ where $G=\Z^n_2$ and the face ring $\Z_2[Q]$ under some conditions on $Q$.
  In the following, we use the notations from Section~\ref{Subsec:mod2-GKM}. \n

  First of all, we have a lemma that is parallel to~\cite[Lemma 6.3]{MasPanov06}.
  
  \begin{lem} \label{Lem:tau-relation}
    For any faces $f$ and $f'$ of $Q$, the relation below holds in $\widehat{H}^*_G(W;\Z_2)$:
    \[ \tau_f\tau_{f'} = \tau_{f\vee f'}\cdot \sum_{f''\in f\cap f'} \tau_{f''}.\]
  Here we define $\tau_{\varnothing}=0$.
  \end{lem}
  \begin{proof}
     The proof is parallel to the proof of~\cite[Lemma 6.3]{MasPanov06}. The idea is to use the 
    monomorphism $r:  \widehat{H}^*_G(W;\Z_2)
     \rightarrow H^*_G(W^G;\Z_2)$ to map both sides of the identity to the fixed points and then use the formula~\eqref{Equ:res-tau} to check that they are equal. 
  \end{proof}
  
  By Lemma~\ref{Lem:tau-relation}, we obtain a well-defined homomorphism 
  \begin{align*} \label{Equ:varphi}
   \varphi: \Z_2[Q] &\longrightarrow \widehat{H}^*_G(W;\Z_2).\\
      v_f\ &\longmapsto \tau_f
  \end{align*}
  
  The following lemma and its proof are parallel to~\cite[Lemma 6.4]{MasPanov06}.
  
  \begin{lem} \label{Lem:injective}
  The homomorphism $\varphi$ is injective if every
  face $Q$ has a vertex.
  \end{lem}
  \begin{proof}
    According to the definitions of $r$ and $s$ (see~\eqref{Equ:Local-map} and~\eqref{Equ:s-p}),
     we have $s=r\circ\varphi$ by identifying
     $H^*_G(p,\Z_2)$ with $\Z_2[Q]\big\slash
    (v_f: p\notin f)$ for every vertex $p$ of $Q$. Then by Lemma~\ref{Lem:res-s},
     $s$ is injective if every
    face of $Q$ has a vertex, so is $\varphi$.
  \end{proof}

  The following lemma is parallel to~\cite[Proposition 7.4]{MasPanov06}. 
  
  \begin{lem} \label{Lem:Cohom-Hat-Gene}
  If the $1$-skeleton of every face of $Q$ (including $Q$ itself) is connected, then $\widehat{H}^*_G(W;\Z_2)$ is generated by the elements $\tau_{F_1},\cdots,\tau_{F_m} \in H^1_G(W;\Z_2)$ as an $H^*(BG;\Z_2)$-module, where $F_1,\cdots, F_m$ are all the facets of $Q$.
  \end{lem}
  \begin{proof}
  The argument is a bit technical, but it is completely parallel to the proof of~\cite[Proposition 7.4]{MasPanov06}. The main idea of the proof is to consider the restriction of an element $\eta\in H^*_G(W;\Z_2)$ 
    to the fixed point set $W^G$ via
     $r: H^*_G(W;\Z_2) \rightarrow H^*_G(W^G;\Z_2)$, and then use $\tau_{F_1},\cdots,\tau_{F_m}$ and
 elements in $H^*(BG;\Z_2)$ to spell out $r(\eta)$ at each  fixed point $p\in W^G$ (see Proposition~\ref{Prop:Equiv-Struc}). The details of the proof are left to the reader.   
   \end{proof}
  
   The following theorem is parallel 
   to~\cite[Theorem 7.5]{MasPanov06}.    
   
   \begin{thm} \label{Thm:Local-Std-Cohom}
    Let $W$ be a locally standard $2$-torus manifold
    with orbit space $Q$. If every face $f$ of $Q$ has a vertex and the $1$-skeleton of $f$ is connected, then
  the map  $\varphi: \Z_2[Q] \rightarrow \widehat{H}^*_G(W;\Z_2)$ is an isomorphism of graded rings.    
   \end{thm}
   \begin{proof}
    By Lemma~\ref{Lem:injective}, $\varphi$ is injective and, by Lemma~\ref{Lem:Cohom-Hat-Gene}, $\varphi$ is surjective.
   \end{proof} 
   
   \begin{lem} \label{Lem:Equiv-1-Skeleton-Conn}
    Let $W$ be an equivariantly formal $2$-torus manifold  
   with orbit space $Q$. Then
  the $1$-skeleton of every face of $Q$ (including $Q$ itself) is connected. 
 \end{lem}
  \begin{proof}
  Since $W$ is equivariantly formal, the localization 
    theorem (Theorem~\ref{Thm:Local-Borel})
     implies that the restriction homomorphism $r: H^*_{G}(W;\Z_2)  \rightarrow H^*_{G}(W^{G};\Z_2)$ is injective. In addition,  since $W$ is connected,  
the image of $H^0_{G}(W;\Z_2)$ under the restriction homomorphism is isomorphic to $\Z_2$. So the ``if'' part of Theorem~\ref{thm:compute} implies that the $1$-skeleton of $Q$ must be connected.\n

For any proper face $f$ of $Q$, the facial submanifold 
  $W_f$ is also an equivariantly formal $2$-torus manifold by Proposition~\ref{Prop:Weak-Facial-2-torus}. Then by applying the above argument to $W_f$, we obtain that the $1$-skeleton of $f$ is also connected. 
 \end{proof}
    
  \begin{cor} \label{Cor:Equiv-Cohom-Face-Ring}
   If $W$ is an equivariantly formal $2$-torus manifold, then the map $\varphi: \Z_2[Q] \rightarrow H^*_G(W;\Z_2)$ is an isomorphism of graded rings.
  \end{cor}
 \begin{proof}
   Since $W$ is equivariantly formal, its equivariant cohomology
   $H^*_G(W;\Z_2)$ is a free module over $H^*(BG;\Z_2)$. So by definition,
  $\widehat{H}^*_G(W;\Z_2) = H^*_G(W;\Z_2)$. 
  For any face $f$ of $Q$,
 the facial submanifold $W_f$ is also an
 equivariantly formal $2$-torus manifold by
 Proposition~\ref{Prop:Weak-Facial-2-torus}. This implies that $f$ has a vertex. 
 Moreover, the $1$-skeleton of $f$ is connected by Lemma~\ref{Lem:Equiv-1-Skeleton-Conn}.
  Then the corollary follows from Theorem~\ref{Thm:Local-Std-Cohom}.
 \end{proof}
 
 When a $2$-torus manifold $W$ is equivariantly formal, Corollary~\ref{Cor:Equiv-Cohom-Face-Ring} tells us that 
 the equivariant cohomology ring of $W$ is completely determined by the face poset of its orbit space (so independent on the characteristic function $\lambda_W$ or the principal bundle
 $\xi_W$). This suggests that the orbit space of $W$ should be rather special.  \n 
 
     The following corollary is parallel to~\cite[Corollary 7.8]{MasPanov06}. It generalizes the calculation of the mod $2$ cohomology ring of a small cover in~\cite{DaJan91}.
     
      \begin{cor}
      If a $2$-torus manifold $W$ is equivariantly formal, then
      \[ H^*(W;\Z_2)\cong \Z_2[v_f: f\ \text{a face of $Q$}] \big\slash I \]
      where $I$ is the ideal generated by the following two types of elements:\n      
            \quad \ $\mathrm{(a)}$ $v_fv_{f'} - v_{f\vee f'}\sum_{f''\in f\cap f'} v_{f''}$, \ \ $\mathrm{(b)}$ $\sum^m_{i=1}\langle t,a_i\rangle v_{F_i}$, $t\in H^1(BG;\Z_2)$.\n       
            
       Here $F_1,\cdots,F_m$ are all the facets of $Q$ and the 
       elements $a_i\in H_1(BG;\Z_2)$ are defined in
       Proposition~\ref{Prop:Equiv-Struc}.     
    \end{cor}
    \begin{proof}
  Since $W$ is equivariantly formal,
    $\iota_{W}^* : H^*_G(W;\Z_2)\rightarrow
 H^*(W;\Z_2)$ is surjective and its kernel is generated by
 all $\rho^*(t)$ with $t\in H^1(BG;\Z_2)$ (see~\eqref{Equ:rho-star}). Then the statement follows from Corollary~\ref{Cor:Equiv-Cohom-Face-Ring} and Proposition~\ref{Prop:Equiv-Struc}.
    \end{proof}

  \subsection{Real blow-up of a locally standard $2$-torus manifold along a facial submanifold}    \label{Subsec:Blow-up}
 \ \n
 
 Let $W$ be a locally standard $2$-torus manifold with orbit space $Q$. 
For a codimension-$k$ face $f$ of $Q$, the facial submanifold $W_f$ is an embedded connected codimension-$k$ submanifold of $W$. So the equivariant normal bundle $\nu_f$ of $W_f$ in $W$ is a real vector bundle of rank $k$.
 If we replace $W_f \subset W$ by the real projective bundle $P(\nu_f)$, we obtain a new $2$-torus manifold 
 denoted by $\widetilde{W}^f$ called the \emph{real blow-up} of $W$ along $W_f$. This is analogous to 
 the blow-up of a torus manifold in~\cite[Sec.\,9]{MasPanov06} (also see~\cite[p.\,605]{GriHarr78} and~\cite[Sec.\,4]{Franz17}).\n
 
  The orbit space of $\widetilde{W}^f$, denoted by $Q^f$, is the result of ``cutting off'' the face
 $f$ from $Q$ (see Figure~\ref{Fig:Blow-up}). So $\widetilde{W}^f$ is also locally standard.
 Correspondingly, the simplicial cell complex $|\mathcal{P}_{Q^f}|$ is obtained from
 $|\mathcal{P}_{Q}|$ by a stellar subdivision of the face dual to $f$.  \n
 
  \begin{figure}
        \begin{equation*}
        \vcenter{
            \hbox{
                  \mbox{$\includegraphics[width=0.5\textwidth]{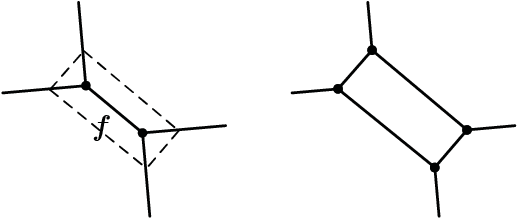}$}
                 }
           }
     \end{equation*}
   \caption{Cutting off a face from a nice manifold with corners} \label{Fig:Blow-up}
   \end{figure}
 
 \begin{prop} \label{Prop:blow-up-Equiv}
 Let $W$ be a locally standard $2$-torus manifold with orbit space $Q$ and $f$ be a proper face of $Q$ with codimension-$k$.
  Then $\widetilde{W}^f$ is equivariantly formal if and only if so is $W$.
 \end{prop}
 \begin{proof}
 (a) Let $\widetilde{\nu}_f$ denote the equivariant normal bundle of $P(\nu_f)$ in $\widetilde{W}^f$. Besides, let $\mathrm{Th}(\nu_f)$ and $\mathrm{Th}(\widetilde{\nu}_f)$ be the Thom space of $\nu_f$ and $\widetilde{\nu}_f$, respectively. Then we have a natural commutative diagram of continuous maps:
  \[ \xymatrix{
           P(\nu_f) \ar[d]_{p_0} \ar[r]^{\  \widetilde{i}}
                & \widetilde{W}^f \ar[d]^{p} \ar[r]^{\widetilde{t}\ \ \ }& \mathrm{Th}(\widetilde{\nu}_f) \ar[d]^q \\
           W_f \ar[r]^{i} & W \ar[r]^{t\ \ \ } & \mathrm{Th}(\nu_f)
                 } 
                 \]
   where $i$ and $\widetilde{i}$ are the inclusions,
   $t$ and $\widetilde{t}$ are the Thom-Pontryagin maps; and $p: \widetilde{W}^f\rightarrow W$ is the blow-down map, $p_0$ is the restriction of $p$ to $P(\nu_f)$ and
    $q$ is the induced map by $p$ in the Thom spaces. 
   
   \n
  According to~\cite[\S 5]{Gilter92} and~\cite[Theorem 3.7]{Gilter92}, there is a short exact sequence:
  \begin{equation} \label{Equ:Gilter-Seq}
   0\longrightarrow H^*(\mathrm{Th}(\nu_f);\Z_2)\overset{\alpha}{\longrightarrow}
    H^*(W;\Z_2)\oplus H^*(\mathrm{Th}(\widetilde{\nu}_f);\Z_2)
   \overset{\beta}{\longrightarrow} H^*(\widetilde{W}^f;\Z_2)\longrightarrow 0. 
   \end{equation}
   where $\alpha=(t^*, q^*)$ and $\beta=p^*-\widetilde{t}^*$.   
    This implies:
    \[ \dim_{\Z_2} H^*(\widetilde{W}^f;\Z_2) = 
    \dim_{\Z_2} H^*(W;\Z_2) +
    \dim_{\Z_2} H^*(\mathrm{Th}(\widetilde{\nu}_f);\Z_2) -
     \dim_{\Z_2} H^*(\mathrm{Th}(\nu_f);\Z_2).
      \]
      
     By the Thom isomorphism, we have
     \begin{align*}
       \dim_{\Z_2} H^*(\mathrm{Th}(\nu_f);\Z_2) &=
     \dim_{\Z_2} H^*(W_f;\Z_2), \\ 
     \dim_{\Z_2} H^*(\mathrm{Th}(\widetilde{\nu}_f);\Z_2)&=
     \dim_{\Z_2} H^*(P(\nu_f);\Z_2).  
     \end{align*}
     
    By Leray-Hirsch theorem, 
     $H^*(P(\nu_f);\Z_2)\cong H^*(W_f;\Z_2)\otimes H^*(\R P^{k-1};\Z_2)$ (as $\Z_2$-vector spaces), which implies
      $ \dim_{\Z_2} H^*(P(\nu_f);\Z_2) = k \cdot \dim_{\Z_2} H^*(W_f;\Z_2)$.     
      So  
      \begin{equation} \label{Equ:Dim-blow-up}
        \dim_{\Z_2} H^*(\widetilde{W}^f;\Z_2) = 
    \dim_{\Z_2} H^*(W;\Z_2) + (k-1) \cdot \dim_{\Z_2} H^*(W_f;\Z_2).
      \end{equation}
      
    If $W$ is equivariantly formal, then $W$ is locally standard and so $Q$ is a nice manifold with corners. It is easy to see
    \[ \# \text{vertices of $Q^f$} = 
      \# \text{vertices of $Q$} + (k-1)\cdot  
      \# \text{vertices of $f$}. \]
      
   Since the fixed point set $W^{G}$ ($G=\Z^n_2$) corresponds to the vertex set of $Q$ which is discrete,
   the number of fixed points of the $G$-action  satisfies 
   \begin{equation} \label{Equ:Fixpoints-blow-up}
     |(\widetilde{W}^f)^G| = |W^G| + (k-1)\cdot |(W_f)^G|.
  \end{equation}  
  
  By Proposition~\ref{Prop:Weak-Facial-2-torus},
  $W_f$ is also equivariantly formal. So by Theorem~\ref{thm:AllPuppe},  
  $$\dim_{\Z_2} H^*(W;\Z_2)=|W^G|, \ \ \dim_{\Z_2} H^*(W_f;\Z_2)=|(W_f)^G|.$$ 
     It follows from \eqref{Equ:Dim-blow-up} and~\eqref{Equ:Fixpoints-blow-up} that 
  $|(\widetilde{W}^f)^G| = \dim_{\Z_2} H^*(\widetilde{W}^f;\Z_2)$.
   So we deduce from Theorem~\ref{thm:AllPuppe} that $\widetilde{W}^f$ is equivariantly formal.
  \n
  
  Conversely, if $\widetilde{W}^f$ is equivariantly formal,
  we have
  \begin{align*}      
    \dim_{\Z_2} H^*(W;\Z_2) &\overset{\eqref{Equ:Dim-blow-up}}{=} \dim_{\Z_2} H^*(\widetilde{W}^f;\Z_2) - (k-1) \cdot \dim_{\Z_2} H^*(W_f;\Z_2)\\
        (\text{by Theorem~\ref{thm:AllPuppe}}) \, &\leq
         |(\widetilde{W}^f)^G|  -  (k-1)\cdot |(W_f)^G| \overset{\eqref{Equ:Fixpoints-blow-up}}{=} |W^G| = 
          \dim_{\Z_2} H^*(W^G;\Z_2).        
  \end{align*}
  
   But by Theorem~\ref{thm:AllPuppe}, $ \dim_{\Z_2} H^*(W^G;\Z_2) \leq  \dim_{\Z_2} H^*(W;\Z_2)$. So we must have $\dim_{\Z_2} H^*(W;\Z_2)= \dim_{\Z_2} H^*(W^G;\Z_2)$, which implies that $W$ is equivariantly formal.    
The proposition is proved.      
 \end{proof}

 The following lemma is parallel to~\cite[Lemma 9.1]{MasPanov06}.  Its proof is almost identical to the proof in~\cite{MasPanov06}, hence omitted.
 
 \begin{lem} \label{Lem:Q-cut-face}
 Let $Q$ be a nice manifold with corners and $f$ be a proper face of $Q$. Then $Q^f$ is mod $2$ face-acyclic if and only if so is $Q$.
 \end{lem}

 \section{Equivariantly formal $2$-torus manifolds with mod $2$ cohomology generated by degree-one part}
 \label{Sec:Equiv-Formal-Cohom-One}
 
  In our study of equivariantly formal $2$-torus manifolds, those manifolds whose mod $2$ cohomology rings are generated by
  their degree-one part are of special importance. 
 We will see in Section~\ref{Sec:Proof-Thm-1} that 
  the study of general equivariantly formal $2$-torus manifolds can be reduced to the study of these special $2$-torus manifolds by a sequence of real blow-ups along
 facial submanifolds. \n
 
The following lemma is parallel to~\cite[Lemma 2.3]{MasPanov06}.
 
 \begin{lem} \label{Lem:Surj-Cohom}
  Suppose there is an equivariantly formal $\Z_2^r$-action on a compact manifold $M$ where the cohomology ring $H^*(M;\Z_2)$ is generated by its degree-one part. Then for any subgroup $H$ of $\Z_2^r$ and every connected component $N$ of $M^H$, the homomorphism $i^*: H^*(M;\Z_2)\rightarrow H^*(N;\Z_2)$ is surjective where $i: N\hookrightarrow M$
  is the inclusion. In particular, $H^*(N;\Z_2)$ is also generated by its degree-one part.
 \end{lem}
 \begin{proof}
  First we assume $H\cong\Z_2$. We have a commutative diagram as follows:
   \begin{equation} \label{Diag:M-N}
     \xymatrix{
           H^*_H(M;\Z_2) \ar[d]_{\iota_M^*} \ar[r]^{\widehat{i}^*_H}
                & H^*_H(N;\Z_2) \ar[d]^{\iota_N^*}  \\
           H^*(M;\Z_2)  \ar[r]^{i^*} & H^*(N;\Z_2) 
                 }  
       \end{equation}
   where 
   $H^*_H(N;\Z_2)\cong H^*(N;\Z_2)\otimes H^*(BH;\Z_2)$
   and $\widehat{i}^*_H$ is the homomorphism on equivariant cohomology induced by $i$.  By our assumption, both $\iota^*_M$ and $\iota^*_N$
   are surjective.
    The following argument is parallel to the proof of~\cite[Lemma 2.3]{MasPanov06}.
    \n
    By~\cite[Theorem VII.1.5]{Bred72}, the inclusion $M^H \hookrightarrow M$ induces an isomorphism $H^k_H(M;\Z_2)\rightarrow
   H^k_H(M^H;\Z_2) $ for sufficiently large $k$, which implies that 
   $$\widehat{i}^*_H:  H^k_H(M;\Z_2)\rightarrow
   H^k_H(N;\Z_2)$$
    is surjective if $k$ is sufficiently large.\n
    
    Let $v_1,\cdots, v_d \in H^1(M;\Z_2)$ be a set of multiplicative generators of
   $H^*(M;\Z_2)$. For each $1\leq l \leq d$,
   let $\widehat{v}_l$ be a lift of
   $v_l$ in $H^*_H(M;\Z_2)$ and $w_l:= i^*(v_l)\in H^1(N;\Z_2)$.
    Let $t$ be a generator of $H^1(BH;\Z_2)\cong \Z_2$.
    By the commutativity of the above diagram~\eqref{Diag:M-N}, 
    $$\widehat{i}^*(\widehat{v}_l) = b_l t+w_l \ \,\text{for some}\ \ b_l\in \Z_2.$$
   
   Then for an arbitrary element $\zeta \in H^*(N;\Z_2)$,
   there exists a large enough integer $q\in \Z$ and a polynomial $P(x_1,\cdots, x_d)$ such that
   $$ \widehat{i}^* \big( P(\widehat{v}_1,\cdots,
   \widehat{v}_d) \big) = \zeta\otimes t^{q} . $$
   
   On the other hand, we have 
    \[ \widehat{i}^* \big( P(\widehat{v}_1,\cdots,
   \widehat{v}_d) \big) = P(b_1 t+w_1,\cdots, 
   b_d t+w_d)=\sum_{k\geq 0} P_k(w_1,\cdots,w_d)\otimes t^k \]
   for some polynomials $P_k$, $k\geq 0$. Hence $\zeta = P_q(w_1,\cdots, w_d) = i^*(P(v_1,\cdots, v_d))$. Therefore, $i^*$ is surjective and $H^*(N;\Z_2)$ is generated by $w_1,\cdots, w_d \in H^1(N;\Z_2)$. \n
   
   For the general case, suppose $H\cong \Z^s_2$, $1\leq s\leq r$. Then we have a sequence:
   \[  \{0\}= H_0\subset H_1\subset H_2 \subset \cdots\subset H_s = H  \]
   where $H_l\cong \Z^l_2$ for each $0\leq l \leq s$. Moreover, we have
   \[ M^H = \big((M^{H_1})^{H_2\slash H_1})\cdots \big)^{H_s\slash H_{s-1}} ,\;  H_l\slash H_{l-1}\cong \Z_2, \, l=1,\cdots, s.\]
    Repeating the above argument for each $H_l\slash H_{l-1}$ proves the lemma.     
  \end{proof}
  
 The following lemma is parallel to~\cite[Lemma 3.4]{MasPanov06}.
 
 \begin{lem} \label{Lem:Conn-Intersect}
 Let $W$ be an equivariantly formal $2$-torus manifold whose cohomology ring
  $H^*(W;\Z_2)$ is generated by its degree-one part.
  Then all non-empty multiple intersections of the characteristic submanifolds of $W$ are equivariantly formal $2$-torus manifolds whose mod $2$ cohomology rings are generated by their degree-one part as well.
 \end{lem}
 \begin{proof}    
    Let $F_1,\cdots, F_m$ be all the facets of $Q$ and $G=\Z^n_2$ where $n=\dim(W)$. 
    In the following, we use the notations defined in Section~\ref{Subsec:mod2-GKM}. 
  First of all, since the characteristic submanifold $W_i$ is a connected component of the fixed point set $X^{G_i}$, Lemma~\ref{Lem:Surj-Cohom} implies that
   the restriction $H^*(W;\Z_2)\rightarrow H^*(W_i;\Z_2)$ is surjective. So the $G$-action on $W_i$ is equivariantly formal (by Proposition~\ref{Prop:Facial-2-torus}). Then we 
   have 
  \begin{align*}
    H^*_G(W;\Z_2) &\cong H^*(W;\Z_2)\otimes H^*(BG;\Z_2), \\
    H^*_G(W_i;\Z_2)&\cong H^*(W_i;\Z_2)\otimes H^*(BG;\Z_2).
    \end{align*}
    It follows that the restriction
  $H^*_G(W;\Z_2)\rightarrow H^*_G(W_i;\Z_2)$ is also surjective.   
  In addition,  
 by using Proposition~\ref{Prop:Equiv-Struc}\,(i) and a completely parallel argument to the proof of~\cite[Prop.\,3.4(2)]{Mas99}, we can prove the following claim.
 \n
 \textbf{Claim:} $H^*_G(W;\Z_2)$ is generated as a ring by
  all the equivariant Thom classes $\tau_1,\cdots,\tau_m$
  of the normal bundles of the characteristic submanifolds $W_1,\cdots, W_m$. \n
  
   When $W_{j_1}\cap\cdots\cap W_{j_s} = \varnothing$,
  $\tau_{j_1}\cdots\tau_{j_s}$ clearly vanishes. So the above claim implies that for any $k\geq 0$, $H^k_G(W;\Z_2)$ is additively generated by the monomials $\tau^{k_1}_{j_1}\cdots\tau^{k_s}_{j_s}$
  such that $W_{j_1}\cap\cdots\cap W_{j_s}\neq \varnothing$ and $k_1+\cdots+k_s =k$.\n
  
  Let $N$ be a connected component of $W_{i_1}\cap\cdots\cap W_{i_k}$, $1\leq k \leq n$. Then $N$ is the facial submanifold $W_f$ over some codimension-$k$ face $f$ of $Q$. So by Lemma~\ref{Lem:Surj-Cohom}, $N$ is an equivariantly formal $2$-torus manifold whose cohomology ring $H^*(N;\Z_2)$ is generated by its degree-one part.
  Moreover, by a completely parallel argument to the proof of~\cite[Lemma 3.4]{MasPanov06}, 
 we can show that $N$ is the only connected component of $W_{i_1}\cap\cdots\cap W_{i_k}$ from the above discussion of $H^k_G(W;\Z_2)$. The lemma is proved. 
     \end{proof}
 
    The following proposition is parallel to~\cite[Lemma 8.2\,(2)]{MasPanov06}.
    
  \begin{prop} \label{Prop:Cohom-Polytope}
        Suppose $W$ is an $n$-dimensional equivariantly formal $2$-torus manifold
  with orbit space $Q$ and the cohomology ring $H^*(W;\Z_2)$ is generated by its degree-one part. Then the geometrical realization 
  $|\mathcal{P}_Q|$ of the
  face poset $\mathcal{P}_Q$ of $Q$ is a Gorenstein* simplicial complex over $\Z_2$. In particular, $\Z_2[\mathcal{P}_Q]=\Z[Q]$ is Cohen-Macaulay and $H^*(|\mathcal{P}_Q|;\Z_2)\cong H^*(S^{n-1};\Z_2)$. 
   \end{prop} 
   \begin{proof}
    By Lemma~\ref{Lem:Conn-Intersect}, all non-empty multiple intersections of the characteristic submanifolds
    of $W$ are connected. This implies that $|\mathcal{P}_Q|$ is
    a simplicial complex. Moreover, by~\cite[II 5.1(d)]{Stanley07},  it is enough to verify the following three conditions to prove that $|\mathcal{P}_Q|$ is Gorenstein*
    over $\Z_2$:
    \begin{itemize}
     \item[(a)] $\Z_2[\mathcal{P}_Q]$ is Cohen-Macaulay;\n
     \item[(b)] Every $(n-2)$-simplex in $\mathcal{P}_Q$ is contained in exactly two $(n-1)$-simplices;\n
      \item[(c)] $\chi(\mathcal{P}_Q)=\chi(S^{n-1})$.
    \end{itemize}  
   
   Since $W$ is equivariantly formal,
   $H^*_G(W;\Z_2)$ is a free $H^*(BG;\Z_2)$-module 
    and   
    $\Z_2[\mathcal{P}_Q]=\Z_2[Q]$ is isomorphic to $H^*_G(W;\Z_2)$ (by Corollary~\ref{Cor:Equiv-Cohom-Face-Ring}) where $G=\Z^n_2$. This implies (a). 
   \n
    
    Note that each $(n-2)$-simplex of $\mathcal{P}_Q$
    corresponds to a non-empty intersection of $n-1$ 
characteristic submanifolds of $W$. The latter intersection is an equivariantly formal $1$-manifold by Lemma~\ref{Lem:Conn-Intersect}, so it is a circle with exactly two $G$-fixed points. This implies (b).\n

The proof of (c) is completely parallel to~\cite[Lemma 8.2\,(2)]{MasPanov06}, so we leave it to the reader. The proposition is proved.
  \end{proof}
   
  Using the above proposition and the lemmas from Section~\ref{Sec-Proof-Main-1}, we obtain the following theorem that is parallel 
  to~\cite[Theorem 7.7]{MasPanov06}. 
   
   \begin{thm} \label{Thm:Equiv-Form-Cohom-Deg1}
    Let $W$ be a $2$-torus manifold whose orbit space is $Q$. Then $W$ is equivariantly formal and the cohomology ring $H^*(W;\Z_2)$ is generated by its degree-one part
    if and only if the following three conditions are satisfied:
       \begin{itemize} 
          \item[(a)] $H^*_G(W;\Z_2)$ is isomorphic to $\Z_2[Q]=\Z_2[\mathcal{P}_Q]$ as a graded ring.\n 
          
     \item[(b)] $\Z_2[Q]$ is Cohen-Macaulay. \n
     
     \item[(c)] $|\mathcal{P}_Q|$ is a simplicial complex.
     \end{itemize}
   \end{thm}
   \begin{proof}
  The argument is completely parallel to the proof of~\cite[Theorem 7.7]{MasPanov06}. We only need to
  replace $T^n$ by $\Z^2_n$ and $\mathbb{Q}$-coefficients by $\Z_2$-coefficients to obtain our proof here.
   \end{proof}

 \section{Proof of Theorem~\ref{thm:Main-1}} \label{Sec:Proof-Thm-1}

 In this section, we give a proof of Theorem~\ref{thm:Main-1}. Our proof
   follows the proof of~\cite[Theorem 8.3, Theorem 9.3]{MasPanov06} almost step by step, while some
   arguments for $2$-torus manifolds here are simpler than those for torus manifolds in~\cite{MasPanov06}. 

 \subsection{Equivariant cohomology of the canonical model} \label{Subsec:Canon-Model} \ \n
 
  Let $Q$ be a connected compact smooth nice $n$-manifold with corners. We call any function
  $ \lambda: \mathcal{F}(Q)\rightarrow \Z^n_2$ that satisfies the linear independence relation 
  in Section~\ref{Subsec:Mfd-corners} a \emph{characteristic function} on $Q$.
 By the same gluing rule in~\eqref{Equ-Glue-Const-Z2}, we
 can obtain a space $M_Q(\lambda)$ from any 
 characteristic function $\lambda$ on $Q$, called the \emph{canonical model} determined by $(Q,\lambda)$. It is easy to see that $M_Q(\lambda)$ is a  $2$-torus manifold of dimension $n$.
  \n
   
   Let $Q^{\vee}$ denote the cone of the geometrical realization of the \emph{order complex} $\mathrm{ord}(\overline{\mathcal{P}}_Q)$ of $\overline{\mathcal{P}}_Q=\mathcal{P}_Q-\{\hat{0}\}$. So
   topologically, $Q^{\vee}$ is homeomorphic to $\text{Cone}(|\mathcal{P}_Q|)$. Moreover,   
  $Q^{\vee}$ is
   a ``space with faces'' (see Davis~\cite[Sec.\,6]{Da83})
   where each proper face $f$ of $Q$ determines 
   a unique ``face'' $f^{\vee}$ of $Q^{\vee}$ that is the geometrical realization of the order complex of the poset 
   $\{ f' \,|\, f' \subseteq f  \}$. More precisely, $f^{\vee}$ consists of all simplices of the form 
   $f'_k\subsetneq\cdots\subsetneq f'_1\subsetneq f'_0=f$
   in $\mathrm{ord}(\overline{\mathcal{P}}_Q)$.
   The ``boundary'' of $Q^{\vee}$, denoted by $\partial Q^{\vee}$, is $\mathrm{ord}(\overline{\mathcal{P}}_Q)$ which is homeomorphic to $|\mathcal{P}_Q|$. So we have
   homeomorphisms:
    \begin{equation} \label{Equ:Q-hat-Bound}
      \partial Q^{\vee} \cong |\mathcal{P}_Q|, \ \ 
    Q^{\vee}\cong \mathrm{Cone}(|\mathcal{P}_Q|). 
    \end{equation}
   
   \begin{rem}
       When $|\mathcal{P}_Q|$ is a simplicial complex, the space $Q^{\vee}$
with the face decomposition was called in~\cite[p.\,428]{DaJan91} a \emph{simple polyhedral complex}. \n
\end{rem}
 
 Suppose $F_1,\cdots, F_m$ are all the facets of $Q$. Let
 $\mathcal{F}(Q^{\vee}) = \{ F^{\vee}_1,\cdots, F^{\vee}_m\}$. Then any characteristic function
 $\lambda: \mathcal{F}(Q)\rightarrow \Z^n_2$ induces a
 map $\lambda^{\vee}: \mathcal{F}(Q^{\vee})\rightarrow \Z^n_2$ where $\lambda^{\vee}(F^{\vee}_i) = \lambda(F_i)$,
 $1\leq i \leq m$. Then by the same gluing rule in~\eqref{Equ-Glue-Const-Z2}, we obtain a   
 space $M_{Q^{\vee}}(\lambda^{\vee})$ with a canonical $\Z^n_2$-action. 
 By the same argument as in the proof of~\cite[Proposition 5.14]{MasPanov06}, we can prove the following.
  
 \begin{prop} \label{Prop:Q-hat-Map}
 There exists a 
   continuous map $\phi: Q\rightarrow Q^{\vee}$ which preserves the face structure and induces an equivariant continuous map
   $$ \Phi: M_Q(\lambda) \rightarrow M_{Q^{\vee}}(\lambda^{\vee}).$$
 \end{prop}
 
 Here $\phi: Q\rightarrow Q^{\vee}$ is constructed inductively, starting from an identification
of vertices and extending the map on each higher-dimensional face by a degree-one
map. Since every face $f^{\vee}$ of $Q^{\vee}$ is a cone,  there are no obstructions to such extensions.\n

 In addition,
 by a similar argument to that in~\cite[Theorem 4.8]{DaJan91}, we can obtain the following result.
 
 \begin{prop}  \label{Prop:Q-hat-Cohom}
   $H^*_G(M_{Q^{\vee}}(\lambda^{\vee});\Z_2)$ is isomorphic to $\Z_2[Q]$ where $G=\Z^n_2$.
 \end{prop} 
 
 On the other hand, 
 $H^*_G(M_Q(\lambda);\Z_2)$ could be much more complicated. Indeed, it is shown in~\cite[Theorem 1.7]{Yu20} that
 $H^*_G(M_Q(\lambda);\Z_2)$ isomorphic to the so called
  \emph{topological face ring} of $Q$ over $\Z_2$ which involves the mod $2$ cohomology rings of all the
  faces of $Q$.

 \subsection{Proof of Theorem~\ref{thm:Main-1}\,(ii)}

 \begin{proof}
   We first prove the ``if'' part. Let $Q$ be an $n$-dimensional mod $2$ homology polytope and $G=\Z^n_2$. Since $H^1(Q;\Z_2)=0$ and $W$ is locally standard, the principal $G$-bundle $\xi_W$ determined by $W$ is a trivial $G$-bundle over $Q$. Then by~\cite[Lemma 3.1]{MaLu08}, $W$ is equivariantly homeomorphic to the canonical model $M_Q(\lambda_W)$ (see~\eqref{Equ-Glue-Const-Z2}).
  So by Proposition~\ref{Prop:Q-hat-Cohom}, there exists an 
  equivariant continuous map
  $$\Phi: W =M_Q(\lambda_W) \rightarrow M_{Q^{\vee}}(\lambda^{\vee}_W) := W^{\vee}.$$
  
 Let $\pi: W\rightarrow Q$ and $\pi^{\vee}: W^{\vee}\rightarrow Q^{\vee}$ be the projections, respectively.
 Let $F_1,\cdots, F_m$ be all the facets of $Q$. Since $Q$ is a mod $2$ homology polytope, so are $F_1,\cdots, F_m$. 
 For brevity, let 
 $$W_i = \pi^{-1}(F_i),\ \ W^{\vee}_i = (\pi^{\vee})^{-1}(F^{\vee}_i), \ 1\leq i \leq m.$$ 
 
 It is easy to see that the $\Z^n_2$-actions on $W\backslash \bigcup_i W_i$
 and $W^{\vee}\backslash \bigcup_i W^{\vee}_i$ are both
 free. Then we have
 $$  H^*_G\Big(W,\bigcup_i W_i;\Z_2 \Big) \cong
 H^*(Q,\partial Q;\Z_2), \ \ 
   H^*_G\Big(W^{\vee},\bigcup_i W^{\vee}_i;\Z_2\Big)
   \cong H^*(Q^{\vee},\partial Q^{\vee};\Z_2).  $$
 
 So $\Phi: W \rightarrow W^{\vee}$ induces a map between the following two exact sequences:
 \begin{equation} \label{Equ:Phi-Diagram}
       \xymatrix{
        \ \ar[r] &   H^*(Q^{\vee},\partial Q^{\vee};\Z_2)\ar[d]^{\phi^*} \ar[r]
        & H^*_G(W^{\vee};\Z_2)  \ar[d]^{\Phi^*} \ar[r] & H^*_G \big(\bigcup_i W^{\vee}_i;\Z_2 \big) \ar[d]^{\Phi^*} \ar[r] & \cdots  \\
         \ \ar[r]  &   H^*(Q,\partial Q;\Z_2)  \ar[r] &
          H^*_G(W;\Z_2)  \ar[r] & 
          H^*_G\big(\bigcup_i W_i;\Z_2 \big) \ar[r] & \cdots
                 } 
 \end{equation}
 
  Each $W_i$ is a $2$-torus manifold over the homology polytope $F_i$. So using induction and a Mayer-Vietoris argument, we may assume that in the diagram~\eqref{Equ:Phi-Diagram}, 
  $\Phi^* : H^*_G(\bigcup_i W^{\vee}_i;\Z_2)\rightarrow 
  H^*_G(\bigcup_i W_i;\Z_2)$ is an isomorphism.\n
    
  By Proposition~\ref{Prop:Hom-Polytope}, $H^*(|\mathcal{P}_Q|;\Z_2)\cong H^*(S^{n-1};\Z_2)$. Then by~\eqref{Equ:Q-hat-Bound}, we obtain
   \[ H^*(Q^{\vee},\partial Q^{\vee};\Z_2)\cong 
   H^*(D^n,S^{n-1};\Z_2).  \] 
  
  We also have $H^*(Q,\partial Q;\Z_2)\cong 
   H^*(D^n,S^{n-1};\Z_2)$ since $Q$ is an $n$-dimensional mod $2$ homology polytope. By the construction of $\phi$, it is easy to see that the homomorphism
    $\phi^*: H^*(Q^{\vee},\partial Q^{\vee};\Z_2)\rightarrow
    H^*(Q,\partial Q;\Z_2)$ is an isomorphism.
    Then by applying the five-lemma to the diagram~\eqref{Equ:Phi-Diagram}, we can deduce that $\Phi^*: H^*_G(W^{\vee};\Z_2)
    \rightarrow  H^*_G(W;\Z_2)$ is also an isomorphism.
   So by Proposition~\ref{Prop:Q-hat-Cohom}, $H^*_G(W;\Z_2)\cong \Z_2[Q]$.\n
   
    Besides, we also know that $\Z_2[Q]$
    is Cohen-Macaulay by Proposition~\ref{Prop:Hom-Polytope}. Then since $|\mathcal{P}_Q|$
   is a simplicial complex, all the three conditions in  Theorem~\ref{Thm:Equiv-Form-Cohom-Deg1} are satisfied. Hence
  $W$ is equivariantly formal
      and $H^*(W;\Z_2)$
     is generated by its degree-one part as a ring. The ``if'' part is proved.
  \n
  
   Next, we prove the ``only if'' part.  By the assumption on $W$ and Lemma~\ref{Lem:Conn-Intersect},
    all non-empty multiple
intersections of characteristic submanifolds of $W$ are connected and their cohomology
rings are generated by their degree-one elements. 
 So we may assume by induction that
all the proper faces of $Q$ are mod $2$ homology polytopes. In particular, the proper faces of $Q$ are all
 mod $2$ acyclic. From these assumptions, we need to prove that $Q$ itself is mod $2$ acyclic. \n
 
 By Proposition~\ref{Prop:Cohom-Polytope}, $|\mathcal{P}_Q|$ is a simplicial complex. So $|\mathcal{P}_Q|$ is the nerve simplicial complex of the cover of $\partial Q$ by the facets of $Q$.
 By a Mayer-Vietoris sequence argument, we can deduce that
 $H^*(\partial Q;\Z_2)\cong H^*(|\mathcal{P}_Q|;\Z_2)$. 
 This together with Proposition~\ref{Prop:Cohom-Polytope} shows that
  \begin{equation} \label{Equ:Bound-Q}
     H^*(\partial Q;\Z_2)\cong H^*(S^{n-1};\Z_2).
   \end{equation}  
  \n
   
   \textbf{Claim:} $H^1(Q;\Z_2)=0$.\n
      
   Since $W$ is equivariantly formal, $H^*_G(W;\Z_2)$ is a free $H^*(BG;\Z_2)$-module. On the other hand,
   $H^*(Q,\partial Q;\Z_2)$ is finitely generated
   over $\Z_2$ since $Q$ is compact. So 
     $H^*(Q,\partial Q;\Z_2)$ is a torsion $H^*(BG;\Z_2)$-module.
   It follows that the whole bottom row in the diagram~\eqref{Equ:Phi-Diagram} splits into short exact sequences:
   \begin{equation} \label{Equ:Short-Exact-Seq}
     0\rightarrow  
          H^k_G(W;\Z_2)  \rightarrow
          H^k_G \Big(\bigcup_i W_i;\Z_2 \Big) \rightarrow
          H^{k+1}(Q,\partial Q;\Z_2) \rightarrow 0, \ k\geq 0.
   \end{equation}
   
   Take $k=0$ above, we 
   clearly have $H^0_G(W;\Z_2)\cong  H^0_G \Big(\bigcup_i W_i;\Z_2 \Big)\cong \Z_2$. This implies 
  $H^1(Q,\partial Q;\Z_2)=0$. So in the following exact sequence,
  \[ \cdots \rightarrow H^1(Q,\partial Q;\Z_2)\rightarrow H^1(Q;\Z_2) \rightarrow H^1(\partial Q;\Z_2) \rightarrow \cdots, \]
  $H^1(Q;\Z_2)$ is mapped injectively into $H^1(\partial Q;\Z_2)\cong H^1(S^{n-1};\Z_2)$. Note that if $n=1$, the claim is trivial.
  When $n=2$, we have $\partial Q =S^1$ and $H^1(Q;\Z_2) = 0$ or $\Z_2$. 
  But by the classification of compact surfaces,
  the latter case is impossible. When $n\geq 3$, we have
  $H^1(\partial Q;\Z_2)=0$, so $H^1(Q;\Z_2)=0$.
  The claim is proved.\n
  
  Now since $H^1(Q;\Z_2)=0$, by the above proof of the ``if'' part, there exists an equivariant homeomorphism $\Phi$ from $W$ to the canonical model $M_Q(\lambda_W)$. In addition, by~\eqref{Equ:Q-hat-Bound} and Proposition~\ref{Prop:Cohom-Polytope}, we have
  $$H^*(\partial Q^{\vee};\Z_2) \cong H^*( |\mathcal{P}_Q|;\Z_2) \cong H^*(S^{n-1};\Z_2).$$  
 So we have an isomorphism
   \begin{equation} \label{Equ:Q-vee}
     H^*(Q^{\vee},\partial Q^{\vee};\Z_2)\cong H^*(D^n,S^{n-1};\Z_2).
     \end{equation}
   
  Then by the construction of $\phi$, the map $\phi^*: H^*(Q^{\vee},\partial Q^{\vee};\Z_2) \rightarrow H^*(Q,\partial Q;\Z_2)$ is an isomorphism in degree $n$ (since $Q$ is connected) and thus is injective in all degrees. So by an extended
version of the $5$-lemma, we can deduce that in the diagram~\eqref{Equ:Phi-Diagram} the map $\Phi^*:
H^*_G(W^{\vee};\Z_2)\rightarrow H^*_G(W;\Z_2)$ is injective. Moreover,
 \begin{itemize}
   \item $H^*_G(W^{\vee};\Z_2)=H^*_G(M_{Q^{\vee}}(\lambda^{\vee}_W);\Z_2)\cong \Z_2[Q]$
by Proposition~\ref{Prop:Q-hat-Cohom}, and\n
\item
$\Z_2[Q]\cong H^*_G(W;\Z_2)$ by Corollary~\ref{Cor:Equiv-Cohom-Face-Ring}.
\end{itemize}
 So $H^*_G(W^{\vee};\Z_2)$ and $H^*_G(W;\Z_2)$ have the same dimension over $\Z_2$ in each degree.
Therefore, the monomorphism $\Phi^*:
H^*_G(W^{\vee};\Z_2)\rightarrow H^*_G(W;\Z_2)$ is actually an isomorphism. Then by the $5$-lemma again, we can deduce from the diagram~\eqref{Equ:Phi-Diagram} that $\phi^*: H^*(Q^{\vee},\partial Q^{\vee};\Z_2)\rightarrow
    H^*(Q,\partial Q;\Z_2)$ is an isomorphism.
   So by~\eqref{Equ:Q-vee}, 
   $$  H^*(Q,\partial Q;\Z_2)\cong H^*(D^n,S^{n-1};\Z_2) $$
   which implies that $Q$ is mod $2$ acyclic by Poincar\'e-Lefschetz duality.
   This finishes the proof. 
 \end{proof}

 \subsection{Proof of Theorem~\ref{thm:Main-1}\,(i)}  
    
  \begin{proof} We can reduce Theorem~\ref{thm:Main-1}\,(i) to Theorem~\ref{thm:Main-1}\,(ii) by real blow-ups of $W$ along sufficient many facial submanifolds, which corresponds to doing some barycentric subdivisions of
the face poset $\mathcal{P}_Q$ of $Q$ (see Figure~\ref{Fig:Cutting-Face}). Indeed, after doing enough barycentric subdivisions to $\mathcal{P}_Q$,
we can turn $|\mathcal{P}_Q|$ into a simplicial complex.
Let $\widehat{W}$ be the $2$-torus
manifold obtained after these real blow-ups on $W$ and $\widehat{Q}$
be its orbit space (with $|\mathcal{P}_{\widehat{Q}}|$
being a simplicial complex).

 \begin{itemize}  
  \item[Fact-1:]  $\widehat{W}$
  is equivariantly formal if and only if so is $W$ (by Proposition~\ref{Prop:blow-up-Equiv}).
  
  \item[Fact-2:] $\widehat{Q}$ is mod $2$ face-acyclic if and only if so is $Q$ (by Lemma~\ref{Lem:Q-cut-face}).\n
  
   \end{itemize}
   
    \begin{figure}
        \begin{equation*}
        \vcenter{
            \hbox{
                  \mbox{$\includegraphics[width=0.86\textwidth]{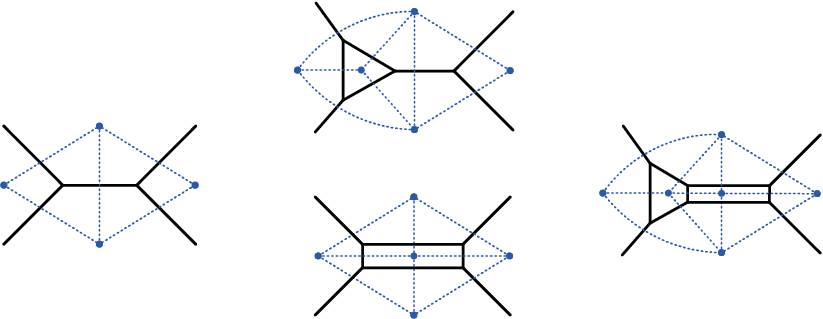}$}
                 }
           }
     \end{equation*}
   \caption{Cutting a vertex and an edge} \label{Fig:Cutting-Face}
   \end{figure}

    We first prove the ``if'' part. Suppose $W$ is locally standard and $Q$ is mod $2$ face-acyclic. Then $\widehat{W}$
   is also locally standard and $\widehat{Q}$ is a mod $2$ homology polytope by Fact-2. 
   So by Theorem~\ref{thm:Main-1}\,(ii), $\widehat{W}$ is equivariantly formal, then so is $W$.
 \n
 
  Next, we prove the ``only if'' part. If $W$ is equivariantly formal, then so is $\widehat{W}$, and $W$ is locally standard by Theorem~\ref{Thm:Locally-Std}. So by Corollary~\ref{Cor:Equiv-Cohom-Face-Ring}, we have a graded ring isomorphism $H^*_G(\widehat{W};\Z_2)\cong \Z_2[\widehat{Q}]$.
 Moreover, since $|\mathcal{P}_{\widehat{Q}}|$ is a simplicial complex, $\Z_2[\widehat{Q}]$ is generated by its degree-one elements, then so is $H^*_G(\widehat{W};\Z_2)$. In addition, since $\iota_{\widehat{W}}^* : H^*_G(\widehat{W};\Z_2)\rightarrow
 H^*(\widehat{W};\Z_2)$ is surjective,
   $H^*(\widehat{W};\Z_2)$ is also generated by its degree-one elements. Then by Theorem~\ref{thm:Main-1}\,(ii), $\widehat{Q}$ is a mod $2$ homology polytope. So by Fact-2, $Q$ is mod $2$ face-acyclic.  
   \end{proof}

   \subsection{Proof of Theorem~\ref{thm:Main-2}}
 \ \n
\begin{proof} 
  We first prove the ``if'' part.  
   Assume that there exists a regular $\mathrm{m}$-involution $\tau$ on $W$. By definition the fixed point set $W^{\tau}$ of $\tau$ is discrete, then 
   so is $W^{\Z^n_2}\subseteq W^{\tau}$.
 This implies that $Q$ must have vertices.   
     Let $p$ be a vertex of $Q$ and let
     $F_1,\cdots, F_n$ be all the facets containing $p$.
     By the property of $\lambda_W$, 
     $$e_1=\lambda_W(F_1),\cdots, e_n=\lambda_W(F_n)$$ form a linear basis of $\Z_2^n$ over $\Z_2$. Then since the $\Z_2^n$-action on $W$ is locally standard, it is easy to see that only when $g=e_1+\cdots+e_n$ could the fixed point set $W^{\tau_g}$ be discrete. So we must have $\tau=\tau_{e_1+\cdots+e_n}$, and in particular
     $$W^{\tau}=W^{\tau_{e_1+\cdots+e_n}}=W^{\Z_2^n}.$$ 
  Hence   $$ \dim_{\Z_2} H^*(W^{\Z_2^n};\Z_2) = \dim_{\Z_2} H^*(W^{\tau};\Z_2) =\dim_{\Z_2} H^*(W;\Z_2)$$
    where the second ``$=$'' is due to the assumption that $\tau$ is an $\mathrm{m}$-involution.  
 So by Theorem~\ref{thm:AllPuppe}, $W$ is equivariantly formal. Then $Q$ is mod $2$ face-acyclic by Theorem~\ref{thm:Main-1}. In particular, 
    every face of $Q$ has a vertex and the $1$-skeleton of $Q$ is connected (by Lemma~\ref{Lem:1-Skeleton-Connect}).\n
     
 It remains to prove that
       the image of $\lambda_W: \mathcal{F}(Q)\rightarrow \Z_2^n$ is exactly $\{e_1,\cdots, e_n\}$.
       Indeed, take an edge $e$ of $Q$ whose vertices are $p$ and $p'$. 
    So the $n$ facets of $Q$ that contain $p'$ are $F_1,\cdots, F_{i-1}, F'_i, F_{i+1},\cdots, F_n$ for some $1\leq i \leq n$. Then since
     $\tau_{e_1+\cdots+e_n}$ is an $\mathrm{m}$-involution, we must have
     $$ \lambda_W(F_1)+\cdots + \lambda_W(F_{i-1}) +\lambda_W( F'_i)+ \lambda_W(F_{i+1}) + \cdots+\lambda_W(F_n) = e_1+\cdots+e_n. $$
     This implies $\lambda_W( F'_i)=e_i$. Then since the $1$-skeleton of $Q$ is connected and every facet $F$ of $Q$ contains a vertex, we can iterate the above argument
  to prove that every 
 $\lambda_W(F)$ must take value in $\{e_1,\cdots, e_n\}$.
 \n
  Next, we prove the ``only if'' part. Suppose
       $Q$ is mod $2$ face-acyclic and the values of
      the characteristic function $\lambda_W$ 
       of $Q$ consist exactly of a linear basis $e_1,\cdots, e_n$ of $\Z^n_2$.
    By Theorem~\ref{thm:Main-1}\,(i), $W$ is equivariantly formal. So we have
      $$  \dim_{\Z_2} H^*(W^{\Z_2^n};\Z_2)=\dim_{\Z_2} H^*(W;\Z_2) \ \ \text{(by Theorem~\ref{thm:AllPuppe})}.$$
            
        On the other hand, our assumption on $\lambda_W$  implies that the regular involution 
       $\tau=\tau_{e_1+\cdots+e_n}$ satisfies $W^{\tau}=W^{\Z^n_2}$ which is a discrete set. Then we have
       $$ \dim_{\Z_2} H^*(W^{\tau};\Z_2) = \dim_{\Z_2} H^*(W^{\Z_2^n};\Z_2)=\dim_{\Z_2} H^*(W;\Z_2).$$
      So $\tau$        
       is a regular $\mathrm{m}$-involution on $W$ by definition. The theorem is proved.
        \end{proof}
 
        \begin{rem}
    If we do not assume a $2$-torus manifold $W$ to be locally standard, even if $W$ admits a regular $\mathrm{m}$-involution, $W$ may not be  equivariantly formal or locally standard. For example: let 
    $$S^2=\{(x_1,x_2,x_3)\in \R^3\,|\, x_1^2+x^2_2 + x^2_3 =1 \}.$$
    
     Define two involutions
    $\sigma$ and $\sigma'$ on $S^2$ by
     $$\sigma(x_1,x_2,x_3)=(-x_1,-x_2,x_3), \quad
      \sigma'(x_1,x_2,x_3) = (x_1,x_2,-x_3). $$
      
    It is easy to see that $\sigma$ is an 
    $\mathrm{m}$-involution on $S^2$ with two isolated fixed points $(0,0,1)$ and $(0,0,-1)$.
      But since the $\Z^2_2$-action on $S^2$ determined by $\sigma$ and $\sigma'$ has no global fixed point, it is not equivariantly formal. We can also directly check that this $\Z^2_2$-action on $S^2$ is not locally standard. 
   \end{rem} 
   
  Finally, we propose some questions on weakly equivariantly formal $2$-torus manifolds.\n  
   
   \textbf{Question-3: Does there exist a weakly equivariantly formal $2$-torus manifold which is
    not equivariantly formal?}
   \n
   
   \textbf{Question-4:} If a $2$-torus manifold is weakly equivariantly formal, are there any restrictions on
  the topology and combinatorial structure of its orbit space?\n  
  \textbf{Question-5:} Whether or not a $2$-torus manifold being weakly equivariantly formal is determined only by the topology and combinatorial structure of its orbit space?

    \section*{Acknowledgment}
   The author wants to thank Anton Ayzenberg and the anonymous reviewer for some valuable comments and suggestions.

   \section*{Conflict of interest}
           
 The author declares that he has no conflict of interest.

\end{document}